\pdfoutput=1

\documentclass[a4paper]{article}

\usepackage[shortlabels]{enumitem}
\usepackage[colorlinks,
linkcolor=blue,
anchorcolor=blue,
citecolor=blue]{hyperref}
\usepackage{array}
\usepackage{multirow}
\usepackage{amsmath}
\usepackage{amssymb}
\usepackage{stmaryrd}
\usepackage{graphicx}
\usepackage{wasysym}
\usepackage{diagbox}
\usepackage{lipsum}
\usepackage{amsthm}
\usepackage{float}
\usepackage{hyperref}
\usepackage{graphicx}
\usepackage{placeins}
\usepackage{makecell}
\usepackage{xr}

\theoremstyle{definition}
\newtheorem{thm}{\bf Theorem}

\newtheorem{defn}{\bf Definition}

\newtheorem{prop}{\bf Proposition}
\newtheorem{cor}{\bf Corollary}
\newtheorem{rem}{\bf Remark}
\newtheorem{obs}{\bf Observation}

\usepackage[english]{babel}

\begin{document}

\title{On Rigid Origami \uppercase\expandafter{\romannumeral2}: Quadrilateral Creased Papers}

\author{
	Zeyuan He$^{1}$, Simon D. Guest$^{2}$}

\author{Zeyuan He, Simon D. Guest\footnote{zh299@cam.ac.uk, sdg@eng.cam.ac.uk. Department of Engineering, University of Cambridge, Cambridge CB2 1PZ, United Kingdom.}}
\maketitle

\begin{abstract}
\noindent Miura-ori is well-known for its capability of flatly folding a sheet of paper through a tessellated crease pattern made of repeating parallelograms. Many potential applications have been based on the Miura-ori and its primary variations. Here we are considering how to generalize the Miura-ori: what is the collection of rigid-foldable creased papers with a similar quadrilateral crease pattern as the Miura-ori? This paper reports some progress. We find some new variations of Miura-ori with less symmetry than the known rigid-foldable quadrilateral meshes. They are not necessarily developable or flat-foldable, and still only have single degree of freedom in their rigid folding motion. This article presents a classification of the new variations we discovered and explains the methods in detail.
\end{abstract}

\begin{small}
	{$\; \; \,$ \textbf{Keywords}:} rigid-foldable, folding, degree-4, Kokotsakis quadrilateral
\end{small}


	
\section{Introduction}

The most broadly studied rigid origami tessellation is the Miura-ori \cite{miura_method_1985-2} introduced by Prof. Kyoto Miura as an efficient packing method for large membranes in space, as shown in figure \ref{fig: miura}. It has a periodical crease pattern, where each unit cell consists of a degree-4 vertex connecting four parallelograms. Miura-ori can be folded where all deformation is concentrated on the rotation of perfectly rigid panels around creases. This property is called the \textit{rigid-foldability}. The \textit{rigid folding motion} starts from the planar state to the flat-folded state where, globally, there is an in-plane shrinkage deformation. Up to now there has been many successful applications inspired by the rigid-foldability of Miura-ori and its primary variations, such as a lithium-ion battery that can also be largely folded, bended and twisted \cite{song_origami_2014}; a flat-foldable corrugated vault used as transformable architecture that connects two
existing buildings \cite{tachi_freeform_2010}; a compliant mechanism used for energy
absorption and impact force distribution \cite{tolman_elastic_2017}, etc. Encouraged by these applications, we are motivated to explore more rigid-foldable variations of the Miura-ori with similar crease patterns consisting of quadrilaterals. Hence here we raise a question: What is the collection of rigid-foldable creased papers with a similar quadrilateral crease pattern as the Miura-ori? This article makes some progress on classifying such rigid-foldable quadrilateral creased papers by introducing new tools on the study of rigid-foldability. We believe that this theoretical progress will cast light on more engineering applications where transformable piecewise-rigid structures are needed.

\begin{figure} [!tb]
	\noindent \begin{centering}
		\includegraphics[width=1\linewidth]{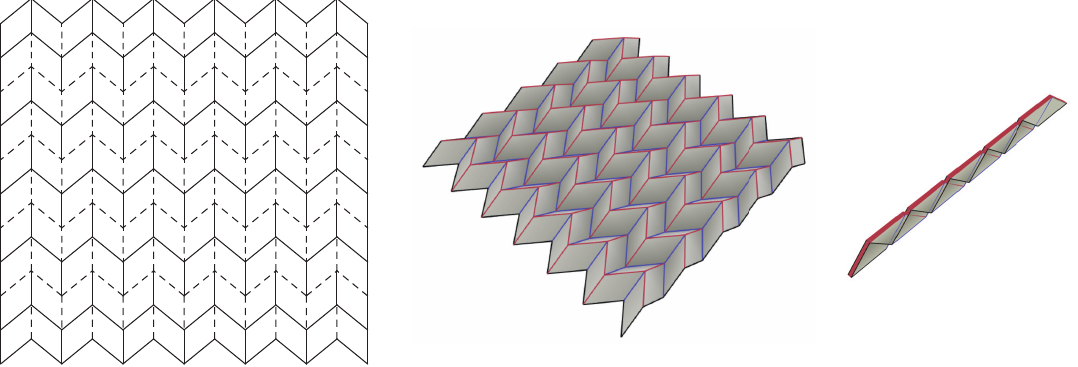}
		\par \end{centering}
	
	\caption{\label{fig: miura}We show the crease pattern, an intermediate rigidly folded state, and the flat-folded state of a Miura-ori. Near each vertex, there are three mountain folds (coloured red or solid lines) and one valley fold (coloured blue or dashed lines), or vice versa. The rigid folding motion of a Miura-ori is an in-plane shrinkage deformation, which is plotted by an open source software called Freeform Origami \cite{tachi_generalization_2009}.}
\end{figure}

There are a number of previous results concerning rigid-foldable quadrilateral creased papers. The majority of them focus on a \textit{developable} (can be folded planar) and \textit{flat-foldable} (can be folded flat) quadrilateral creased paper, which is an important subset of the quadrilateral creased papers considered here. Tachi \cite{tachi_generalization_2009} first described the sufficient and necessary condition for a developable and flat-foldable quadrilateral creased paper to be rigid-foldable, and some other papers followed with further examples and design methods \cite{evans_rigidly_2015, lang_rigidly_2018}. If we go out of this subset, \cite{tachi_freeform_2010} introduced a specific type of non-developable and flat-foldable quadrilateral creased paper, while \cite{stachel_remarks_2011, lang_twists_2017, dieleman_jigsaw_2020} introduced some developable and non-flat-foldable quadrilateral creased papers. Here our target is to provide a full description of rigid-foldable quadrilateral creased papers, without any restriction on the developability or flat-foldability. 


In the first paper of this series \cite{he_rigid_2019}, we presented more details for all necessary definitions used in the current paper, and we mentioned an interesting topic called \textit{generic rigid-foldability}: given a crease pattern $C$, whether (almost) all possible creased papers with crease pattern isomorphic to $C$ are rigid-foldable. The crease pattern of the Miura-ori turns out to be generically rigid, since there exists a first-order rigid quadrilateral creased paper. What we have done in this paper is exactly to find the non-generic quadrilateral creased papers that are rigid-foldable, which is challenging due to the complicated compatibility conditions over the whole creased paper.

To find variations of Miura-ori without strong symmetry, our method is to apply the result in \cite{izmestiev_classification_2016}, where a nearly complete classification of rigid-foldable $3 \times 3$ quadrilateral mesh (or called Kokotsakis quadrilateral) is presented, and the constraints on sector angles are given symbolically. Since a large quadrilateral mesh is rigid-foldable if and only if its restriction on each $3 \times 3$ mesh is rigid-foldable, by considering proper stitching of these Kokotsakis quadrilaterals it would be possible to find all analytical variations of a large rigid-foldable quadrilateral mesh. Although our generalization is still incomplete, this stitching method is shown effective and has enabled new analytical design of rigid-foldable quadrilateral meshes. Further, this stitching method has the potential of being applied to other crease patterns or origami structures with each inner vertex degree-4 but each inner panel not quadrilateral, such as the Kagome pattern, origami stacked metamaterials consisting of several layers \cite{schenk_geometry_2013}, origami ring structures \cite{tachi_rigid-foldable_2012}, etc.

From the next section we will start to discuss the collection of rigid-foldable quadrilateral creased papers. In order to present further results some preliminaries are needed.

\section{Preliminaries}

\subsection{Quadrilateral and degree-4 single creased paper}

We first present some definitions related to a quadrilateral creased paper.

\begin{defn} \label{defn: quad}
	A graph is \textit{isomorphic} to another graph if they have the same number of vertices which are also connected in the same way. The detailed definition of a \textit{creased paper} is provided in \cite{he_rigid_2019}, roughly speaking, it is a polygonal mesh in $\mathbb{R}^3$ that allows contact of different parts but does not allow self-intersection. A \textit{quadrilateral} creased paper is a creased paper with a crease pattern isomorphic to the $m \times n$ ($m,n$ are positive integers) planar quadrilateral mesh shown in figure \ref{fig: introduction}(a). A vertex or crease is called \textit{inner} if it is not on the boundary of a creased paper, otherwise \textit{outer}. A panel (face of a creased paper) is called \textit{inner} if none of its vertices is on the boundary of a creased paper, otherwise \textit{outer}.
\end{defn}

\begin{rem}
	When talking about isomorphism we do not consider the correspondence of mountain-valley assignment, since for a general quadrilateral mesh the mountain-valley assignment is not applicable to distinguish different branches of rigid folding motion, which will be explained at the end of this section. Further, a creased paper is neither necessarily developable or flat-foldable. A \textit{developable} creased paper has a \textit{planar} rigidly folded state where all the folding angles are zero. A \textit{flat-foldable} creased paper has a \textit{flat} rigidly folded state where all the folding angles are $\pm \pi$.
\end{rem}

\begin{rem} \label{rem: not including ring}
	Apart from the creased paper mentioned in Definition \ref{defn: quad}, there might be other creased papers with a crease pattern isomorphic to other quadrilateral meshes. An example is shown in figure \ref{fig: introduction}(c). In this "ring" structure the rigid-foldability condition is more complicated than the quadrilateral creased paper studied here. When generalizing the stacking \cite{schenk_geometry_2013} or cylindrical \cite{tachi_rigid-foldable_2012} structures that are widely used for meta-materials, the analysis of such ``ring'' structure is indeed necessary. It would be discussed in a future article in this series.
\end{rem}

\begin{figure} [!tb]
	\noindent \begin{centering}
		\includegraphics[width=1\linewidth]{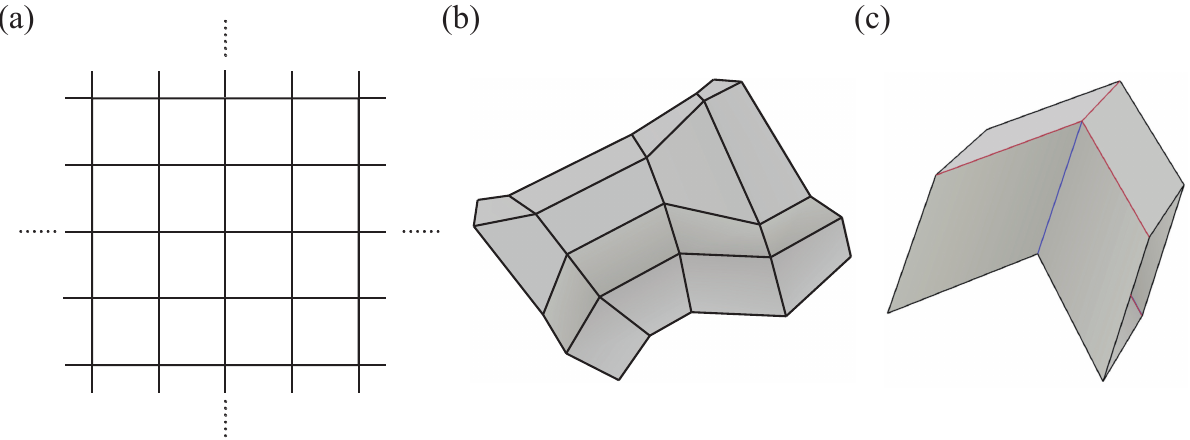}
		\par \end{centering}
	
	\caption{\label{fig: introduction}(a) shows a $m \times n$ ($m,n$ are positive integers) planar quadrilateral mesh, and our object of study is a creased paper with a crease pattern isomorphic to (a). A quadrilateral creased paper is not necessarily developable or flat-foldable, and in (b) a non-developable and non-flat-foldable quadrilateral creased paper is shown. (c) is an example of a ``ring'' structure mentioned in Remark \ref{rem: not including ring}, plotted by Freeform Origami \cite{tachi_generalization_2009}. The mountain and valley creases are coloured red and blue. Such system will not be discussed in this article.}
\end{figure}

In order to study the rigid-foldability of a quadrilateral creased paper, we first need to clarify the relation among the folding angles around its each single-vertex. The restriction of a creased paper on a degree-4 vertex and its adjacent panels is called a degree-4 \textit{single-vertex creased paper}. As shown in figure \ref{fig: degree-4 vertex}, a single-vertex creased paper has four sector angles $\alpha_1,\alpha_2,\alpha_3,\alpha_4$ and four folding angles $\rho_1,\rho_2,\rho_3,\rho_4$. We provide a complete result for the relation among the folding angles of a degree-4 single-vertex creased paper in Section \ref{S-section: s1} of the supplementary material, where all the relations are expressed in terms of $\tan \dfrac{\rho_i}{2}$ in a compact form. Note that the sector angles are assumed in $0 < \alpha_i < \pi$. This is reasonable because if a sector angle equals to $\pi$ (without loss of generality, suppose $\alpha_1=\pi$), then $\alpha_2+\alpha_3+\alpha_4 \ge \pi$;
\begin{enumerate} [(1)]
	\item if $\alpha_2+\alpha_3+\alpha_4 = \pi$, $\rho_1=\rho_4$; $\rho_2 = \rho_3 = 0$.
	\item otherwise, $\rho_2, \rho_3 = \{a,b\}$ or $\{-a,-b\}$, where $a,b$ are non-zero constants.
\end{enumerate}
both of which can be regarded as folding along a crease. Additionally, if a sector angle is greater than $\pi$ (without loss of generality, suppose $\alpha_1>\pi$), there is no essential difference if substituting $\alpha_1$ with $2\pi-\alpha_1$, which means for a set of sector angles $0 < \alpha_i < \pi$, switching $\alpha_i$ to $2\pi-\alpha_i$ will only result in greater possibility of self-intersection of panels.

\begin{figure} [!tb]
	\noindent \begin{centering}
		\includegraphics[width=1\linewidth]{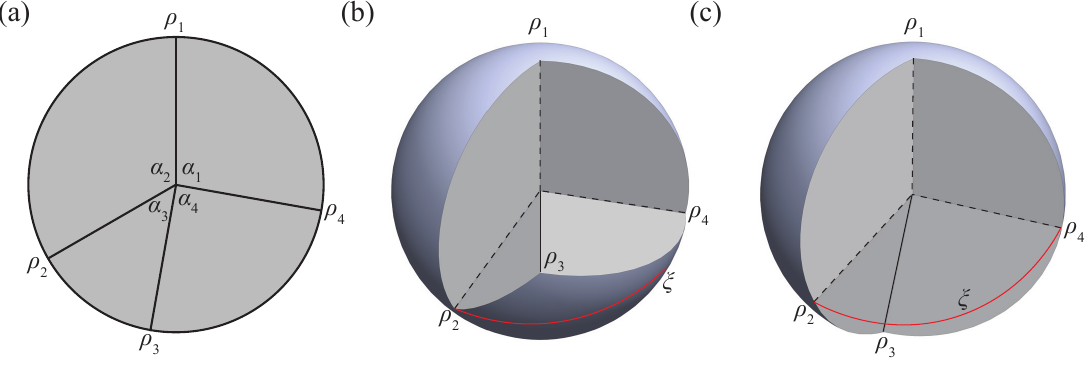}
		\par \end{centering}
	
	\caption{\label{fig: degree-4 vertex}(a) A degree-4 single-vertex creased paper. We label the sector angles counterclockwisely. (b) and (c) show two non-trivial rigidly folded states with the outside edges of the single-vertex creased paper drawn on a sphere as arcs of great circles, assuming the panel corresponding to $\alpha_1$ is fixed when changing the magnitude of $\rho_1$. The mountain and valley creases are shown in solid and dashed lines. Generically, for a $\rho_1$ there are two sets of folding angles $\rho_2, \rho_3, \rho_4$, where the panels corresponding to $\alpha_3$ and $\alpha_4$ are symmetric to $\xi$.}
\end{figure}

\subsection{Rigid-foldability of a quadrilateral creased paper}

Here we will introduce the condition (also called the loop condition \cite{tachi_generalization_2009}) for a quadrilateral creased paper to be rigid-foldable. We start with a special case, where each inner vertex is degree-4, but there is no cycle in the interior of the crease pattern.

\begin{defn}
	In graph theory, a \textit{forest} is a disjoint union of trees. A \textit{tree} is an undirected graph in which any two vertices are connected by exactly one path.
\end{defn}

\begin{prop} \label{thm: no panel}
	(theorem 8 in \cite{he_rigid_2019}) If a creased paper satisfies (figure \ref{fig: no inner panel}(a) and \ref{fig: no inner panel}(b)):
	\begin{enumerate} [(1)]		
		\item The interior of crease pattern is a forest.
		\item The restriction of a rigidly folded state on each single creased paper is rigid-foldable.
	\end{enumerate} 
	then this rigidly folded state is generically rigid-foldable. Especially, if the creased paper described above is also developable, the planar folded state is rigid-foldable.
\end{prop}

\begin{figure} [!tb]
	\noindent \begin{centering}
		\includegraphics[width=1\linewidth]{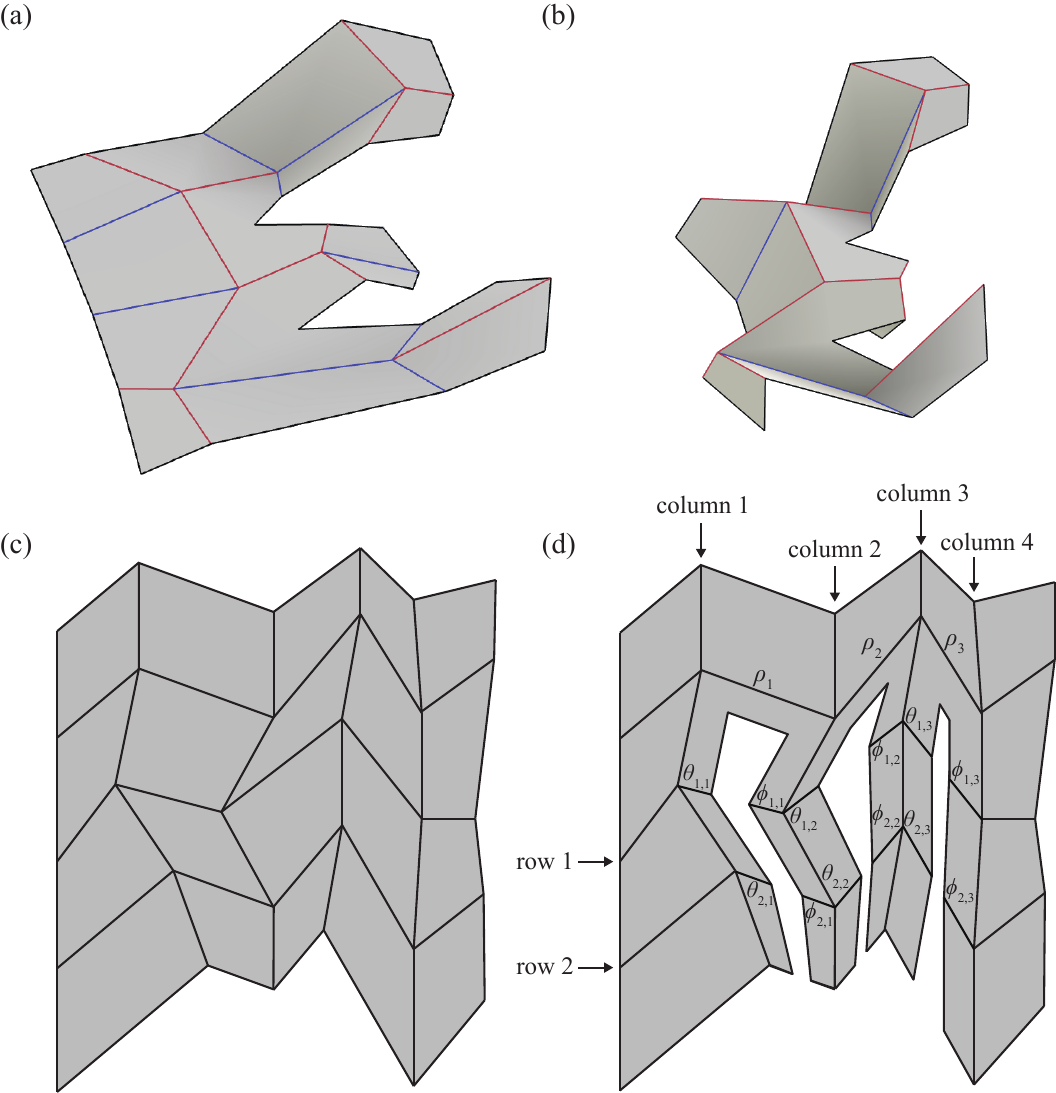}
		\par \end{centering}
	
	\caption{\label{fig: no inner panel}(a) An example of the creased paper described in proposition \ref{thm: no panel}. (b) shows a rigidly folded state of (a), where some panels clash and the rigid folding motion halts. The mountain and valley creases are coloured red and blue. (a) and (b) are plotted by Freeform Origami \cite{tachi_generalization_2009}. Given a quadrilateral creased paper in (c), (d) demonstrates its corresponding decoupled creased paper, where we ``cut'' inner creases connecting adjacent columns to make the interior of crease pattern have no cycle. $\rho_j$ are the folding angles of the top row.}
\end{figure}

Regardless of the details on genericity, it is reasonable that the decoupled crease pattern mentioned in proposition \ref{thm: no panel} can generate rigid-foldability. In the next proposition we will show that the rigid-foldability of the creased paper above is closely related to the rigid-foldability of a quadrilateral creased paper. 

\begin{prop} \label{thm: rigid-foldability}	
	Consider a quadrilateral creased paper (figure \ref{fig: no inner panel}(c)). If and only if
	\begin{enumerate} [(1)]
		\item the corresponding decoupled creased paper (figure \ref{fig: no inner panel}(d)) is rigid-foldable.
		\item for all $i, j$, the following equation is satisfied simultaneously in a closed interval
		\begin{equation} \label{eq: blanket}
		\theta_{i,j} \equiv \phi_{i,j}
		\end{equation}
	\end{enumerate}
	then the quadrilateral creased paper is rigid-foldable. Figure \ref{fig: no inner panel}(d) defines the folding angles $\theta_{i,j}$ and $\phi_{i,j}$.
\end{prop}

\begin{proof}
	As it is straightforward to move between the quadrilateral creased paper and the tree structure by cutting or gluing, the sufficiency and necessity are natural.
\end{proof}

From the analysis above, the rigid folding motion of a quadrilateral creased paper is the same as the decoupled creased paper once proposition \ref{thm: rigid-foldability} is satisfied. Based on the information in Section \ref{S-section: s1} of the supplementary material, the rigid folding motion of a degree-4 single-vertex creased paper can generically be parametrized by a single variable, hence the rigid folding motion of a quadrilateral creased paper can also generically be parametrized by a single variable. We can calculate the \textit{degree of freedom} algebraically from the dimension of the tangent space of a given configuration. For a quadrilateral creased paper, the degree of freedom is generically 1, but might be greater than 1 at some special points in the configuration space, such as the planar and flat rigidly folded states. 

A quadrilateral creased paper may have a number of branches of rigid folding motion. For a developable quadrilateral creased paper, we can distinguish different branches of rigid folding motion by different mountain-valley assignments. This is because, for each branch of every degree-4 single-vertex creased paper, the sign of every folding angle remains the same when the rigid folding motion moves away from planar. Each inner vertex should be incident to three mountain creases and one valley crease, or three valley creases and one mountain crease. However, the mountain-valley assignment for different branches of rigid folding motions of a non-developable quadrilateral creased paper does not satisfy the rule above. For example, in a non-developable degree-4 single-vertex creased paper there might be four mountain creases or two mountain and two valley creases (figure \ref{fig: switching a strip}).

In next section we will start constructing more general rigid-foldable quadrilateral creased papers, which is done by "stitching" together rigid-foldable 3 $\times$ 3 quadrilateral meshes.

\section{Stitching rigid-foldable Kokotsakis quadrilaterals} \label{sec: problem setting}

\begin{figure} [!tb]
	\noindent \begin{centering}
		\includegraphics[width=0.9\linewidth]{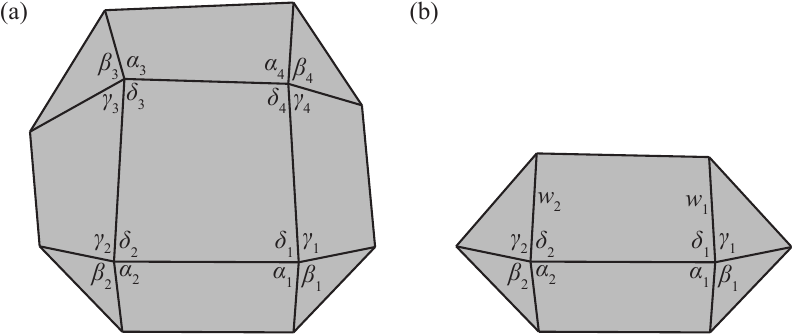}
		\par \end{centering}
	
	\caption{\label{fig: kokotsakis sm}(a) A Kokotsakis quadrilateral. We label these sector angles as $\alpha_j, \beta_j, \gamma_j, \delta_j~(1 \le j \le 4)$. (b) A linear unit that will be introduced in Section \ref{subsection: linear 1}. It is half of a Kokotsakis quadrilateral and the tangent of half of the folding angles satisfies $w_1=cw_2$, $c$ is a constant. If two linear units have the same $c$, they could be stitched together and form a rigid-foldable Kokotsakis quadrilateral. }
\end{figure}

\begin{defn}
	A \textit{Kokotsakis quadrilateral} is a polyhedral surface in $\mathbb{R}^3$, which consists of one centre quadrilateral (the base); four side quadrilaterals, one attached to each side of this centre quadrilateral; and four corner quadrilaterals placed between each two outer consecutive side quadrilaterals (figure \ref{fig: kokotsakis sm}(a)).
\end{defn}

From proposition \ref{thm: rigid-foldability}, the next corollary is natural.

\begin{cor} \label{cor: condition}
	A quadrilateral creased paper is rigid-foldable if and only if the restriction of a rigidly folded state on each Kokotsakis quadrilateral is rigid-foldable.
\end{cor}

Now we explain the process of ``stitching''.

\begin{defn}
	For a Kokotsakis quadrilateral $\rm{Q}_1$, another Kokotsakis quadrilateral $\rm{Q}_2$ can \textit{stitch} with $\rm{Q}_1$ in the longitudinal or transverse direction if $\rm{Q}_2$ can share two vertices with $\rm{Q}_1$ in the longitudinal or transverse direction.
\end{defn}

We provide a nearly complete result on the classification of all rigid-foldable Kokotsakis quadrilaterals \cite{izmestiev_classification_2016} in Section \ref{S-section: s2} of the supplementary material. Each type here is described by a system of equations (most of them are trigonometric, some are exponential or elliptic) on its sector angles. The number of equations $m$ is less than 16, therefore this system of equation can possibly be solved numerically with $16-m$ random input sector angles. However, there are three concerns. First, there might be no solution for sector angles in the real field. Second, even if there is a solution in the real field, the rigid folding motion may locates in the complexified configuration space, hence the Kokotsakis quadrilateral will still be rigid in $\mathbb{R}^3$. Third, even if there is a solution in the real field and there is a rigid folding motion in $\mathbb{R}^3$, the panels may constantly self-intersect, which makes the rigid folding motion not valid.

Then, given any quadrilateral creased paper, we can interpret it by stitching together rigid-foldable Kokotsakis quadrilaterals, either from the same or different types. In other words, if all possible stitching of rigid-foldable Kokotsakis quadrilaterals could be found, we will get the complete description of a rigid-foldable quadrilateral creased paper. In order to grasp the "non-trivial" generalization of the Miura-ori, two restrictions for our object of study are listed below.

\begin{enumerate} [(1)]
	\item There is at least one non-trivial rigid folding motion.

\begin{defn} \label{defn: trivial}
	Following the terminology introduced in \cite{izmestiev_classification_2016}, A rigid folding motion of a creased paper is called \textit{trivial} if a folding angle remains constant during this rigid folding motion. A \textit{trivial} rigid-foldable Kokotsakis quadrilateral only has trivial rigid folding motion.
\end{defn}

We think this requirement is reasonable because when designing a creased paper, a crease that does not participate in the folding seems redundant. If a large quadrilateral creased paper contains a trivial rigid-foldable Kokotsakis quadrilateral, any rigid folding motion of this large quadrilateral creased paper will be trivial, therefore the trivial rigid-foldable Kokotsakis quadrilaterals will not be considered (Type 7 in Section \ref{S-section: s2} of the supplementary material).\\

\item The sector angles can be solved quadrilateral by quadrilateral, i.e. the design of the entire quadrilateral creased paper never requires the solution of equations where variables are sector angles from more than one Kokotsakis quadrilateral.\\	
	
We raise this requirement mainly because even if a quadrilateral creased paper is designed from possible stitchings, generically the number of total constraints will be much more than the number of variables as the size increases. A compromise is making the design process to be as follows. Start from a Kokotsakis quadrilateral. Set its sector angles as known variables and solve the sector angles of the neighbouring Kokotsakis quadrilaterals, quadrilateral by quadrilateral. Continue doing this until all the sector angles are solved. Two advantages of this requirement are:

\begin{enumerate} [(a)]
	\item each step completed by the above process will result in a rigid-foldable creased paper.
	\item when adding a new Kokotsakis quadrilateral, we don't need to adjust the solved part because there is enough design freedom.
\end{enumerate}

\end{enumerate}

In the next sections, we will demonstrate several methods to construct a large rigid-foldable quadrilateral creased paper satisfying the two additional requirements above. We understand that this method will not be comprehensive, i.e. there might be rigid-foldable quadrilateral creased papers whose sector angles can only be solved as a whole, but it is unlikely practicable to find such cases.

\section{Two operations that preserve the rigid-foldability} \label{sec: operation}

Before presenting the catalogue, we provide two operations that preserve the rigid-foldability of a quadrilateral creased paper, named as ``switching a strip'' and "adding a parallel strip".

\subsection{Switching a strip}

\begin{figure} [!tb]
	\noindent \begin{centering}
		\includegraphics[width=1\linewidth]{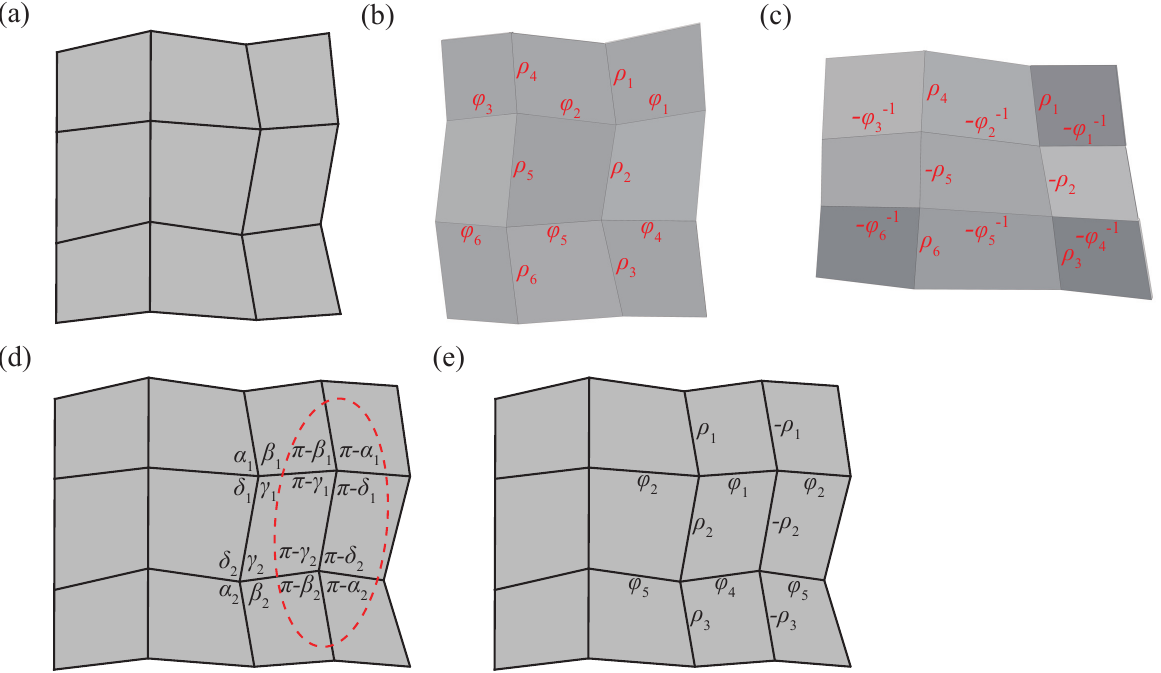}
		\par \end{centering}
	
	\caption{\label{fig: switching a strip}Here we show examples for switching a strip and adding a parallel strip. (a) shows a rigid-foldable and developable quadrilateral creased paper, (b) is a rigidly folded state of (a). We use $\{\rho_1,\rho_2,\rho_3,\rho_4,\rho_5,\rho_6,\phi_1,\phi_2,\phi_3,\phi_4,\phi_5,\phi_6\}$ to represent the tangent of half of the folding angles on these labelled inner creases. In (c), the sector angles on panels of the middle row are replaced by their complements to $\pi$, and after switching a strip $\{\rho_1,\rho_2,\rho_3,\rho_4,\rho_5,\rho_6\} \rightarrow \{\rho_1,-\rho_2,\rho_3,\rho_4,-\rho_5,\rho_6\}$, $\{\phi_1,\phi_2,\phi_3,\phi_4,\phi_5,\phi_6\} \rightarrow \{-\phi_1^{-1},-\phi_2^{-1},-\phi_3^{-1},-\phi_4^{-1},-\phi_5^{-1},-\phi_6^{-1}\}$. Generally, switching a strip will make a developable creased paper non-developable. (d) shows how a parallel strip is added, which is demonstrated with a dashed cycle. In (e), the magnitude of new folding angles are illustrated.}
\end{figure}

\begin{defn} \label{defn: switching}
	\textit{Switching a strip} means to replace all the sector angles on a column or a row of panels by their complements to $\pi$, which is extended from the definition of switching a boundary strip in \cite{izmestiev_classification_2016} (figure \ref{fig: switching a strip}(a--c)).
\end{defn}

\begin{thm} \label{thm: switching}
	Switching a strip preserves the rigid-foldability of a quadrilateral creased paper.
\end{thm}

\begin{proof}
	First consider switching a transverse strip. In figure \ref{fig: switching a strip}, (a) shows a rigid-foldable and developable quadrilateral creased paper, (b) is a rigidly folded state of (a). 
	We choose the middle row of panels to be switched, the sector angles are replaced by their complements to $\pi$. The tangent of half of folding angles on the labelled inner creases are denoted by $\{\rho_1,\rho_2,\rho_3,\rho_4,\rho_5,\rho_6,\phi_1,\phi_2,\phi_3,\phi_4,\phi_5,\phi_6\}$. Following Lemma 4.5 in \cite{izmestiev_classification_2016}, after switching this strip $\{\rho_1,\rho_2,\rho_3,\rho_4,\rho_5,\rho_6\} \rightarrow \{\rho_1,-\rho_2,\rho_3,\rho_4,-\rho_5,\rho_6\}$, $\{\phi_1,\phi_2,\phi_3,\phi_4,\phi_5,\phi_6\} \rightarrow \{-\phi_1^{-1},-\phi_2^{-1},-\phi_3^{-1},-\phi_4^{-1},-\phi_5^{-1},-\phi_6^{-1}\}$, which is shown in figure \ref{fig: switching a strip}(c). For a $m \times n$ quadrilateral mesh, the related folding angles on this strip are changed in the same way. Therefore from proposition \ref{thm: rigid-foldability}, switching a strip preserves the rigid-foldability of a quadrilateral creased paper. The proof for switching a longitudinal strip is similar.
\end{proof}

Note that when switching a strip, the crease lengths might need to be adjusted.

\subsection{Adding a parallel strip}

Given a quadrilateral creased paper, it is always possible to add another row or column of vertices with new inner creases parallel to its adjacent row or column, as shown in figure \ref{fig: switching a strip}(a) and \ref{fig: switching a strip}(d). After adding a parallel strip, the new folding angles are shown in figure \ref{fig: switching a strip}(e), therefore from proposition \ref{thm: rigid-foldability}, adding a parallel strip also preserves the rigid-foldability of a quadrilateral creased paper.

\section{Several types of large rigid-foldable quadrilateral creased papers} \label{section: methods}

In this section we will list several types of rigid-foldable quadrilateral creased papers following the two restrictions mentioned in Section \ref{sec: problem setting}. For all these types, switching some strips and adding some parallel strips leave the type unchanged. The justification for this categorization is provided in Section \ref{section: reason}. All the information on the names of rigid-foldable Kokotsakis quadrilaterals and the relation among sector angles of each type is provided in Section \ref{S-section: s2} of the supplementary material.

\subsection{Orthodiagonal} \label{subsection: orthodiagonal}

This type is stitched from orthodiagonal Kokotsakis quadrilaterals (figure \ref{fig: orthodiagonal}(a)). The independent input sector angles are $\alpha_j$ ($0 \le j \le n$) and $\beta_i$ ($1 \le i \le 2m-1$). The other sector angles $\gamma_{ij}$ can be solved one by one as follows.

\begin{enumerate} [(1)]
	\item In the first row,
	\begin{equation}
	\begin{gathered}
	\cos \alpha_0 \cos \gamma_{11} = \cos \alpha_1 \cos \beta_1 \\
	\cos \alpha_{j-1} \cos \gamma_{1,j} = \cos \alpha_j \cos \gamma_{1,j-1}, ~ 2 \le j \le n
	\end{gathered}
	\end{equation}
	\item From the second row,
	\begin{equation}
	\begin{gathered}
	\dfrac{\tan \beta_{2i-1}}{\tan \beta_{2i}}=\dfrac{\tan \gamma_{2i-1,j}}{\tan \gamma_{2i,j}}, ~ 1 \le i \le m-1, ~ 1 \le j \le n \\
	\cos \beta_{2i} \cos \gamma_{2i+1,1} = \cos \beta_{2i+1} \cos \gamma_{2i,1}, ~ 1 \le i \le m-1 \\
	\cos \gamma_{2i,j-1} \cos \gamma_{2i+1,j} = \cos \gamma_{2i,j} \cos \gamma_{2i+1,j-1}, ~ 1 \le i \le m-1, ~ 2 \le j \le n	
	\end{gathered}
	\end{equation}	
\end{enumerate}

Typical geometric features of this type are:
\begin{enumerate} [(1)]
	\item Each column or row of inner creases are co-planar.
	\item Each plane formed by a column of inner creases is orthogonal to each plane formed by a row of inner creases. Hence the name ``orthodiagonal''.
\end{enumerate}

\begin{rem} \label{rem: orthodiagonal}
	If $\beta_1=\pi-\alpha_0$ and $\beta_{2i+1}=\pi-\beta_{2i}$ ($1 \le i \le m-1$), this quadrilateral creased paper will also be developable. If $\beta_1=\alpha_0$ and $\beta_{2i+1}=\beta_{2i}$ ($1 \le i \le m-1$), this quadrilateral creased paper will also be flat-foldable.
\end{rem}

\begin{figure} [!tb]
	\noindent \begin{centering}
		\includegraphics[width=0.8\linewidth]{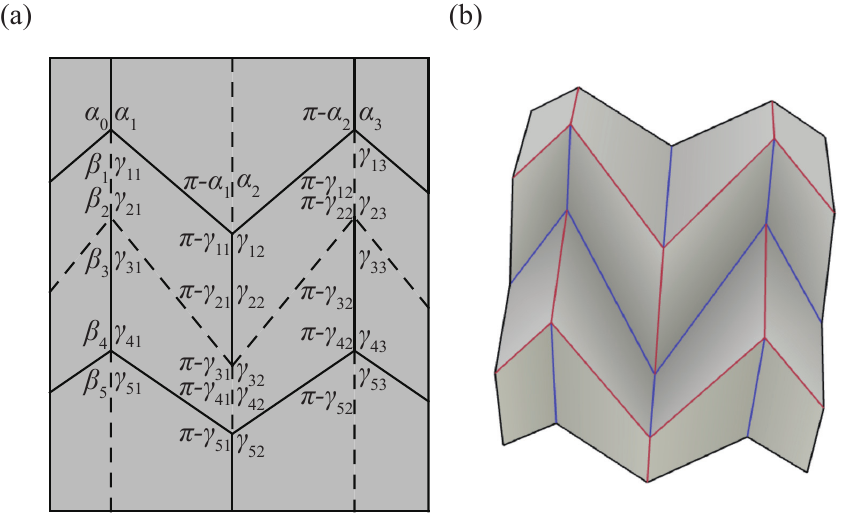}
		\par \end{centering}
	
	\caption{\label{fig: orthodiagonal}An example of the developable case of the orthodiagonal type mentioned in Remark \ref{rem: orthodiagonal}. (a) shows the relation among the sector angles, the mountain and valley creases are shown in solid and dashed lines. (b) is a rigidly folded state of (a). Each column or row of inner creases are co-planar, and each plane formed by a column of inner creases is orthogonal to each plane formed by a row of inner creases. (b) is plotted by Freeform Origami \cite{tachi_generalization_2009}, where the mountain and valley creases are coloured red and blue.}
\end{figure}

\subsection{Isogonal} \label{subsection: isogonal}

This type is the union of flat-foldable quadrilateral creased paper described in \cite{tachi_generalization_2009} and its variation from switching some strips. Typical geometric features of this type are:

\begin{enumerate} [(1)]
	\item If at each inner vertex the sum of opposite sector angles equals to $\pi$, this isogonal type is flat-foldable. Otherwise it is not flat-foldable.
	\item The absolute value of folding angles on a row or column of inner creases are equal.
\end{enumerate}

This type is named isogonal because at every inner vertex opposite sector angles can be equal after switching some strips. Note that the property of ``a rigidly folded state can guarantee a rigid folding motion'' described in \cite{tachi_generalization_2009} is special for the isogonal type because here the existence of a rigidly folded state is equivalent to equation \eqref{eq: blanket}. For other types the rigid-foldability condition on sector angles is more complicated, which prevents similar conclusions.

\subsection{Forward linear repeating} \label{subsection: linear 1}

Before presenting how to construct this type, we introduce \textit{linear units} (figure \ref{fig: kokotsakis sm}(b)) and the \textit{stitching} of linear units. Full information of linear units is provided in Type 4 of Section \ref{S-section: s2} of the supplementary material, and different types of linear units are labelled as 4.1(a), 4.1(b), 4.2(a), 4.2(b), 4.3(a), 4.3(b), 4.4(a), 4.4(b) and 4.5. \textit{Stitching} a linear unit with another linear unit means these two linear units form a Kokotsakis quadrilateral. They do not share vertices, but the folding angles on the matching creases must be the same. Linear unit 1 \textit{can stitch with} linear unit 2 if $\delta_1+\delta_2+\delta_3+\delta_4=2\pi$ and corresponding linear coefficients $c_1=c_2$. If two linear units are stitched in this way, we say they are stitched along the \textit{linear direction}.

\begin{figure} [!]
	\noindent \begin{centering}
		\includegraphics[width=0.9\linewidth]{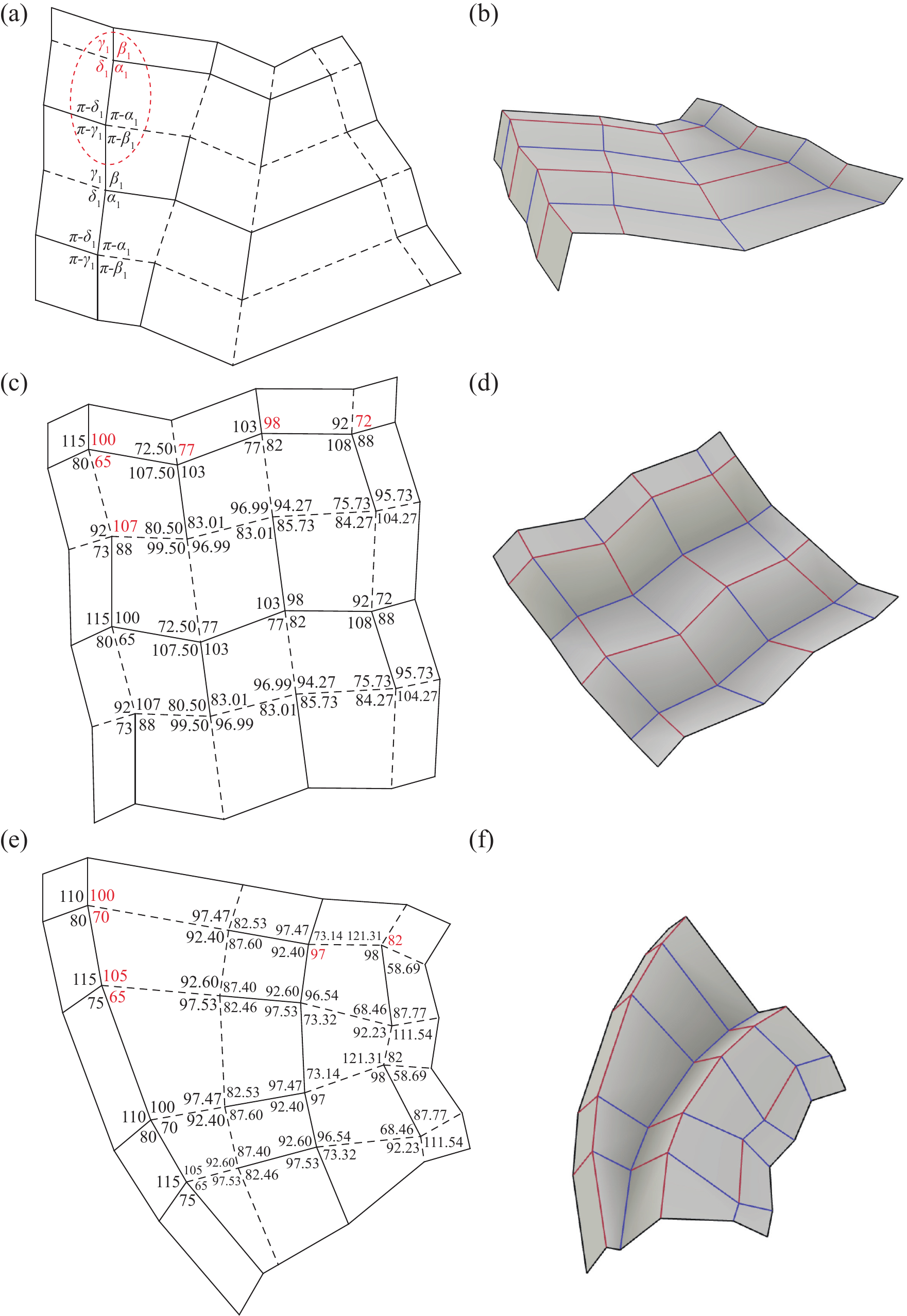}
		\par \end{centering}
	
	\caption{\label{fig: linear repeating}Examples of the forward linear repeating. All the sector angles are labelled in degrees. The independent input sector angles are coloured red in (a), (c) and (e). Other sector angles are solved numerically. The mountain and valley creases are shown in solid and dashed lines. (a) is formed by stitching ``parallel linear units'', one of which is illustrated with a dashed cycle. We name this type as ``parallel repeating''. This is a typical special rigid-foldable quadrilateral creased paper, and some other stitchings of linear units will degenerate to this case. (b) is a rigidly folded state of (a). (c) is formed by 4.1(b)--4.2(b)--4.2(b)--4.1(b). (d) is a rigidly folded state of (c). (e) is formed by 4.1(b)--4.3(b)--4.4(b)--4.1(b). (f) is a rigidly folded state of (e). (b), (d) and (f) are plotted by Freeform Origami \cite{tachi_generalization_2009}, where the mountain and valley creases are coloured red and blue.}
\end{figure}

This type of rigid-foldable quadrilateral creased paper is formed by stitching a row of linear units along the linear direction, and repeating this row orthogonal to the linear direction, which is why it is named "forward linear repeating". To allow repeating, apart from requiring $\delta_1+\delta_2+\delta_3+\delta_4=2\pi$, we should also set $\gamma_1+\gamma_2+\gamma_3+\gamma_4=2\pi$ (figure \ref{fig: kokotsakis sm}(a)). If using linear unit 4.1, $\kappa_1$ and $\kappa_2$ should either both be $\sin\frac{\alpha-\beta}{2}/\sin\frac{\alpha+\beta}{2}$ or both be $\cos\frac{\alpha-\beta}{2}/\cos\frac{\alpha+\beta}{2}$. With the restriction above, the rest of the sector angles can be solved numerically. Some example creased papers are shown in figure \ref{fig: linear repeating}. 

An important notice is there might be ``degeneration'' (which is discussed in Section \ref{section: discussion}) in the solutions. These degenerations can be found both analytically and numerically, and we provide this information in Table \ref{tab: linear degenerations}. Although it will not prevent solutions, the relationships among sector angles will be more special. In order not to make redundant classification, we will not split this type into smaller classes. 

\begin{table} [H]
	\begin{center}
		\begin{tabular}{|c|cccccc|}
			\hline
			& 4.2(a) & 4.2(b) & 4.3(a) & 4.3(b) & 4.4(a) & 4.4(b) \\
			\hline
			4.2(a) & \textcircled{\scriptsize 2} & \textcircled{\scriptsize 2} & \textcircled{\scriptsize 3} & \textcircled{\scriptsize 3} & \textcircled{\scriptsize 1} & \textcircled{\scriptsize 1} \\
			4.2(b) && \textcircled{\scriptsize 2} & \textcircled{\scriptsize 1} & \textcircled{\scriptsize 1} & \textcircled{\scriptsize 1} & \textcircled{\scriptsize 1} \\
			\hline
		\end{tabular}
		\caption{\label{tab: linear degenerations}This table shows when the stitching of linear units degenerates. \textcircled{\scriptsize 1} means the Kokotsakis quadrilateral formed by two linear units will degenerate to a Kokotsakis quadrilateral in the parallel repeating type (figure \ref{fig: linear repeating}(a)). \textcircled{\scriptsize 2} means it will degenerate to a Kokotsakis quadrilateral in the orthodiagonal type (Section \ref{subsection: orthodiagonal}). \textcircled{\scriptsize 3} means both \textcircled{\scriptsize 1} and \textcircled{\scriptsize 2} are possible.}
	\end{center}
\end{table} 

\subsection{Inverse linear repeating} \label{subsection: linear 2}

This type is formed by stitching a row of linear units 4.1 and 4.3 orthogonal to the linear direction, and repeating this row along the linear direction, which is why it is named as "inverse linear repeating". To allow repeating, apart from requiring $\delta_1+\delta_2+\delta_3+\delta_4=2\pi$, we should also set $\alpha_1+\alpha_2+\alpha_3+\alpha_4=2\pi$ (figure \ref{fig: kokotsakis sm}(a)). If using linear unit 4.1, $\kappa_1$ and $\kappa_2$ should either both be $\sin\frac{\alpha-\beta}{2}/\sin\frac{\alpha+\beta}{2}$ or both be $\cos\frac{\alpha-\beta}{2}/\cos\frac{\alpha+\beta}{2}$. An example is shown in figure \ref{fig: hybrid repeating}(a) and \ref{fig: hybrid repeating}(b).

\begin{figure} [!]
	\noindent \begin{centering}
		\includegraphics[width=0.9\linewidth]{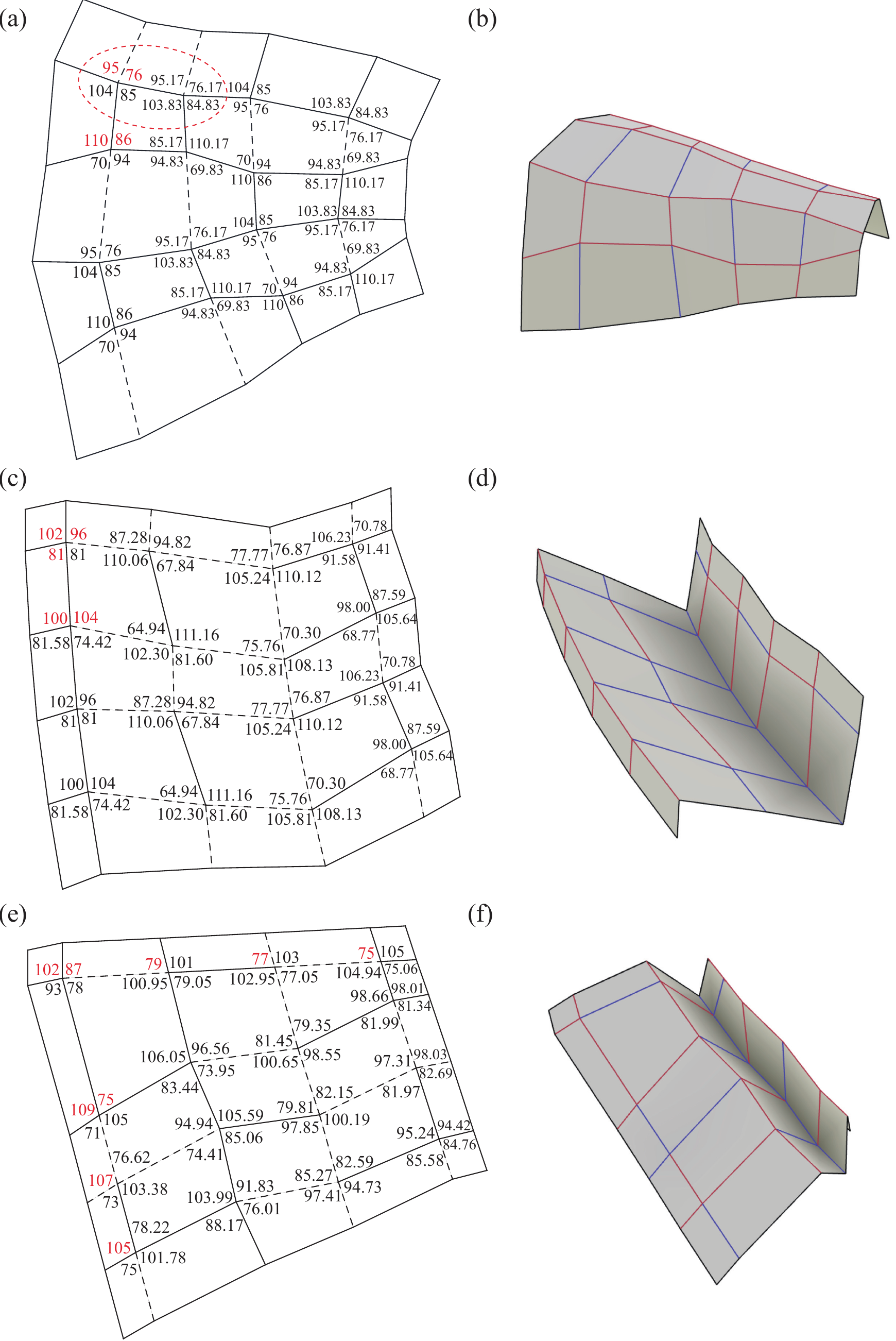}
		\par \end{centering}
	
	\caption{\label{fig: hybrid repeating}Examples of the inverse linear repeating, conic repeating and hybrid type. All the sector angles are labelled in degrees. The independent input sector angles are coloured red. Other sector angles are solved numerically. The mountain and valley creases are shown in solid and dashed lines. (a) is an example of the inverse linear repeating. Here a linear unit 4.1(b) is illustrated with a dashed cycle to show the difference of linear direction in the forward and inverse linear repeating. (b) is a rigidly folded state of (a). (c) is an example of the conic repeating. (d) is a rigidly folded state of (c). (e) is an example of the hybrid type. The first Kokotsakis quadrilateral is of type 6.1, its adjacent row and column are formed by stitching of type 1-2, the rest part is filled with type 1-1. (f) is a rigidly folded state of (e). (b), (d) and (f) are plotted by Freeform Origami \cite{tachi_generalization_2009}, where the mountain and valley creases are coloured red and blue.}
\end{figure}

\subsection{Conic repeating} \label{subsection: conic}

This type is formed by stitching a row of type 2.2 of the rigid-foldable Kokotsakis quadrilaterals, and repeating this row in the longitudinal direction. To allow repeating, apart from requiring $\delta_1+\delta_2+\delta_3+\delta_4=2\pi$, we should also set $\alpha_1+\alpha_2+\alpha_3+\alpha_4=2\pi$ (figure \ref{fig: kokotsakis sm}(a)). An example is shown in figure \ref{fig: hybrid repeating}(c) and \ref{fig: hybrid repeating}(d). 

\subsection{Hybrid} \label{subsection: hybrid}

This type is constructed by the following steps. An example is shown in figure \ref{fig: hybrid repeating}(e) and \ref{fig: hybrid repeating}(f). For the linear compounds type of rigid-foldable Kokotsakis quadrilateral, if it is stitched from two linear units labelled $4.m$ and $4.n$ $(1 \le m \le 5, ~1 \le n \le 5$, either (a) or (b)), we label it as $m$-$n$.

\begin{enumerate} [(1)]
	\item Choose the first Kokotsakis quadrilateral from types 1-1, 5.1, 5.2, 6.2, 6.3, 6.7 or 6.8.
	\item In both the row and column incident to the first Kokotsakis quadrilateral, choose types 1-1, 1-2, 1-3, or 1-4 to stitch with the first Kokotsakis quadrilateral.
	\item Fill the rest of the creased paper with type 1-1. 	
\end{enumerate}

\subsection{Tiling} \label{subsection: tiling}

This type is presented in \cite{dieleman_jigsaw_2020}. It is developable and has 3 independent input sector angles $\alpha, \beta, \gamma$. Setting $\delta=2\pi-\alpha-\beta-\gamma$, at each vertex, the sector angles are $\alpha, \beta, \gamma, \delta$ or $\pi-\alpha, \pi-\beta, \pi-\gamma, \pi-\delta$. We can interpret this type as a combinatorial special solution of the stitching of type 2.2 of the rigid-foldable Kokotsakis quadrilaterals, which is not formed by repeating.

\section{Number of branches of rigid folding motion}

In this section we provide a list of the number of branches of rigid folding motion for the different types listed above. 

\begin{table} [H]
	\begin{center}
		\begin{tabular}{|c|c|}
			\hline
			Type & Number of branches for a $m \times n$ mesh\\
			\hline
			Miura-ori & only upper bound known \cite{ginepro_counting_2014}\\
			\hline
		    Orthodiagonal & 1  \\
			\hline
			Isogonal & 1 \\
			\hline
			Forward linear repeating & from 1 to $2n$ \\
			\hline
			Inverse linear repeating & 1 \\
			\hline
			Conic repeating & 1 \\
			\hline
			Hybrid & 1 \\
			\hline
			Tiling & from 2 to $2^m+2^n-2$ \cite{dieleman_jigsaw_2020}\\
			\hline
		\end{tabular}
	\caption{Number of branches of rigid folding motion for each type of rigid-foldable quadrilateral creased paper presented. For forward linear repeating, linear units 4.1, 4.2 and 4.3 only have one branch, but 4.4 and 4.5 have two branches. The overall number is counted by multiplying the number of branches of each column.}
	\end{center}
\end{table} 

Generally, a rigid-foldable quadrilateral creased paper with more symmetry will have
a greater number of branches. For example, the number for the tiling type and Miura-ori could increase exponentially. However, for other types presented in this article, the number increases linearly or stays at 1.

\section{Justification for the possible types presented}\label{section: reason}

In this section we will describe some of the underlying mathematical structure that leads to the different types presented in Section \ref{section: methods}. 

In Section \ref{S-section: s1} of the supplementary material, a degree-4 single-vertex creased paper is classified as follows, based on different types of the configuration space.

\begin{defn}
	Consider the following equation
	\begin{equation} \label{eq: single-vertex characterization}
	\alpha \pm \beta \pm \gamma \pm \delta=0~(\mathrm{mod}~2\pi)
	\end{equation}
	A degree-4 single-vertex creased paper is said to be
	\begin{enumerate} [(1)]
		\item an \textit{isogram}, if $\alpha=\gamma, \beta=\delta$, and an \textit{antiisogram}, if $\alpha+\gamma=\pi, \beta+\delta=\pi$.
		\item a \textit{deltoid}, if $\alpha=\beta, \delta=\gamma$ or $\alpha=\gamma, \beta=\delta$, and an \textit{antideltoid}, if $\alpha+\beta=\pi, \delta+\gamma=\pi$ or $\alpha+\gamma=\pi, \beta+\delta=\pi$.
		\item of \textit{conic} type, if equation \eqref{eq: single-vertex characterization} has exactly one solution;
		\item of \textit{elliptic} type, if equation \eqref{eq: single-vertex characterization} has no solution;
	\end{enumerate}
\end{defn}

When stitching different types of rigid-foldable Kokotsakis quadrilaterals, vertices that are shared between the quadrilaterals must be of the same type. Note that if two Kokotsakis quadrilaterals cannot stitch with each other, switching a strip cannot make them stitchable since it does not change the type of a vertex, at most switching from deltoid to antideltoid, or isogram to antiisogram. We also need to check whether other relations among the existing sector angles are compatible with the new Kokotsakis quadrilateral we plan to stitch. 

Type 1.1 can only stitch with itself, forming one type of rigid-foldable quadrilateral creased paper, called \textit{orthodiagonal} (Section \ref{subsection: orthodiagonal}). The geometric properties mentioned can be examined directly. We have considered possible rotation of Kokotsakis quadrilaterals, but it does not increase the variation. 

For other types of Kokotsakis quadrilaterals a proposition is provided for further discussion.

\begin{prop} \label{prop: no solution}
	\begin{enumerate} [(1)]
		\item For any type of rigid-foldable Kokotsakis quadrilateral, if the 8 sector angles around two adjacent vertices are given and compatible with the constraints on sector angles, there is enough design freedom to solve the other 8 sector angles.
		\item Only for types 4, 5, 6.2, 6.4, 6.5, 6.7, 6.8 and 6.9, if the given 12 sector angles around three vertices are compatible with the constraints on sector angles, then if the unsolved vertex is an (anti)isogram, there is enough design freedom to solve the other 4 sector angles.
		\item For type 6.6, if the given 12 sector angles around three vertices are compatible with the constraints on sector angles, then if the unsolved vertex is elliptic, there is enough design freedom to solve the other 4 sector angles. 
		\item For other types, even if the given 12 sector angles around three vertices are compatible with the constraints on sector angles. Generically there is no solution for the other 4 sector angles.
	\end{enumerate}
\end{prop}

\begin{proof}
	The equations of rigid-foldability condition on sector angles in each type of 1.2 and 2--6 are independent. If the 8 sector angles around two adjacent vertices are given without contradiction to the constraints on sector angles, the number of remaining equations is no more than 8, and we say there is enough design freedom for the other 8 sector angles. If 12 sector angles are given around three vertices which are compatible with the constraints on sector angles, only for the cases mentioned in (2) and (3) the number of remaining equations is 4, otherwise it is greater than 4. In this sense, we say generically there is no solution for the other 4 sector angles.	
\end{proof}  

For type 1-1 (the labelling is explained in Section \ref{subsection: hybrid}), stitching with itself forms one type of large rigid-foldable quadrilateral creased paper, called \textit{isogonal} (Section \ref{subsection: isogonal}). Other stitching possibilities of type 1-1 are included in the discussions on type 4. 

For other types, proposition \ref{prop: no solution} only provides an estimate. However, this estimate has been verified numerically. Apart from the cases given in proposition \ref{prop: no solution}(2) and \ref{prop: no solution}(3), we have not found it to be possible to obtain a large quadrilateral creased paper just by solving equations quadrilateral by quadrilateral. Therefore our next step is to apply symmetry to construct special solution. Note that from Section \ref{sec: operation}, many other variations based on these special solutions can be obtained by switching some strips and adding some parallel strips.

\subsection{Method 1} \label{spe: repeating}
	
This method is to stitch compatible categories of rigid-foldable Kokotsakis quadrilaterals in one row by solving equations quadrilateral by quadrilateral, and repeating this row in the longitudinal direction to obtain a large quadrilateral creased paper, which makes the creased paper rowwise-periodical. From proposition \ref{prop: no solution} it is possible to solve a row, and when constructing special solutions in the longitudinal direction, we should consider that each time when solving 8 sector angles based on 8 known sector angles, the number of additional constraints to allow such special solutions plus the number of original constraints should not exceed 8. The reason for only using repeating here will be explained in Section \ref{subsection: other}.

Consider generating the row of Kokotsakis quadrilaterals for repeating in the longitudinal direction. Apart from the initial quadrilateral, each time we add a quadrilateral and solve 8 sector angles based on 8 known sector angles. Therefore it is necessary to analyse if the 8 sector angles around two adjacent vertices are given, which type of rigid-foldable Kokotsakis quadrilateral is still rigid-foldable after repeating in the longitudinal direction. Note that those rigid-foldable Kokotsakis quadrilaterals might be able to rotate when stitching. To make it clear, here we describe the placement of a Kokotsakis quadrilateral as ``regular'' if it is aligned with the description in Section \ref{S-section: s2} of the supplementary material, or ``rotated'' if the placement is orthogonal. For regular position, an extra condition $\alpha_1+\alpha_2+\alpha_3+\alpha_4=2\pi$, must be added to allow repetition because the sum of sector angles on each inner panel should be $2\pi$. For rotated position, the extra condition becomes $\gamma_1+\gamma_2+\gamma_3+\gamma_4=2\pi$.

\begin{prop} \label{prop: possible repeating}
	The following statements describe, for each type of rigid-foldable Kokotsakis quadrilateral, whether it is possible to solve 8 sector angles numerically based on 8 known sector angles of adjacent two vertices such that after repeating the new Kokotsakis quadrilateral is still rigid-foldable.	
	\begin{enumerate} [(1)]
		\item For types 2.1 and 2.2, it is possible to obtain another rigid-foldable Kokotsakis quadrilateral by repeating.
		\item For type 1-1, at each vertex there is a coefficient $\kappa$, which can be chosen from $\sin\frac{\alpha-\beta}{2}/\sin\frac{\alpha+\beta}{2}$ or $\cos\frac{\alpha-\beta}{2}/\cos\frac{\alpha+\beta}{2}$. If the number of each choice is 0, 2 or 4, for either the regular or rotated position it is possible to obtain another rigid-foldable Kokotsakis quadrilateral by repeating, otherwise not.
		\item For type 4 (excluding 1-1), for the rotated position it is possible to obtain another rigid-foldable Kokotsakis quadrilateral by repeating. When there is a linear unit 4.1, $\kappa_1$ and $\kappa_2$ should be both  $\sin\frac{\alpha-\beta}{2}/\sin\frac{\alpha+\beta}{2}$ or both $\cos\frac{\alpha-\beta}{2}/\cos\frac{\alpha+\beta}{2}$.		
		\item For type 1-3, for the regular position it is possible to obtain another rigid-foldable Kokotsakis quadrilateral by repeating. When there is a linear unit 4.1, $\kappa_1$ and $\kappa_2$ should be both $\sin\frac{\alpha-\beta}{2}/\sin\frac{\alpha+\beta}{2}$ or both $\cos\frac{\alpha-\beta}{2}/\cos\frac{\alpha+\beta}{2}$. 	
		\item For type 6.1, for either the regular or rotated position, it is possible to obtain another rigid-foldable Kokotsakis quadrilateral by repeating. 
		\item For types 6.4 and 6.5, for the rotated position it is possible to obtain another rigid-foldable Kokotsakis quadrilateral by repeating.
		\item For other types or placements generically it is not possible to obtain another rigid-foldable Kokotsakis quadrilateral by repeating.
	\end{enumerate}
\end{prop}

\begin{proof}
Statements (1), (5) and (6) can be verified directly.

Statements (2), (3) and (4): For type 4, repeating in either the regular or rotated position preserves the linear relations among the tangent of half of the folding angles. For linear unit 4.1, if $\kappa_1$ and $\kappa_2$ are both $\sin\frac{\alpha-\beta}{2}/\sin\frac{\alpha+\beta}{2}$ or both $\cos\frac{\alpha-\beta}{2}/\cos\frac{\alpha+\beta}{2}$, the linear coefficient $c$ remains the same after repeating in either the regular position or the rotated position, otherwise only the sign of $c$ is changed. For linear unit 4.3, $c$ remains the same after repeating in either the regular position or the rotated position. For linear units 4.2--4.5, $c$ remain the same after repeating in the rotated position, while in the regular position change its value. Since when two linear units stitch with each other the linear coefficient must be the same, the proposition holds.

Statement (7): For all the other types or placements, either the additional constraints to allow repeating contradict with the original constraints on sector angles, or the number of additional constraints to allow repeating plus the number of original constraints on sector angles is greater than 8. In this sense we say generically these types or placements are not suitable for repeating.
\end{proof}

The possible types that can be used to stitch with other types also determine the first Kokotsakis quadrilateral in this row. However, our numerical studies have lead to the following restrictive observations about the selection of types.

\begin{obs} \label{obs: negative 1}
	Only non-real numerical solutions are found for types 2.1, 5.3, and 6.9.
\end{obs}

\begin{obs} \label{obs: negative 2}
	The only numerical solution found for type 6.1 degenerates to a Kokotsakis quadrilateral in the ``parallel repeating'' type shown in figure \ref{fig: linear repeating}(a).  
\end{obs}

\begin{obs} \label{obs: negative 3}
	Types 6.4 and 6.5 have numerical solutions, but are rigid in $\mathbb{R}^3$.
\end{obs}

Then considering the compatibility among all the types and placements, the possible stitchings for method 1 are listed as follows. The row of Kokotsakis quadrilaterals to be repeated can be constructed by the following ways. 

\begin{enumerate} [(1)]
	\item \textit{Forward Linear Repeating}: Stitch among the linear units of type 4 in the rotated position (Section \ref{subsection: linear 1}).\\
	
	\item \textit{Inverse Linear Repeating}: Stitch among type 1-3 in the regular position (Section \ref{subsection: linear 2}).\\
	
	Just from the type of vertices in 1-3, we can choose types 1-3, 5.1, 6.3, 6.4 and 6.5 as the first Kokotsakis quadrilateral. However, from Observation \ref{obs: negative 3}, 6.4 and 6.5 are not suitable. More careful examination on types 5.1 and 6.3 shows that the additional requirements to allow repeating contradicts with the rigid-foldability condition on sector angles.\\
	
	\item \textit{Conic Repeating}: Stitch among type 2.2. (Section \ref{subsection: conic})\\
	
	Apart from linear units, from Observations \ref{obs: negative 1}--\ref{obs: negative 3} the only remaining type is 2.2. What we then need to consider is choosing the first Kokotsakis quadrilateral in this row, which can be types 2.2, 5-5, and 6.8. Choosing 5-5 will become a special case of 3.2. More careful examination on type 6.8 shows that the additional requirements to allow repeating contradicts with the rigid-foldability condition on sector angles. 
	~\\
\end{enumerate}    

\subsection{Method 2} \label{spe: solving}

This method is to construct a row and a column of rigid-foldable Kokotsakis quadrilaterals, then consider the stitching of solvable cases mentioned in proposition \ref{prop: no solution}(2) and \ref{prop: no solution}(3) to obtain a large quadrilateral creased paper. It is necessary to make sure that ``solving one vertex based on three vertices'' can proceed in both longitudinal and transverse directions, which excludes type 6.6. Considering solving an (anti)isogram, the possible stitchings are presented in Section \ref{subsection: hybrid}, called the \textit{hybrid} type. From Observations \ref{obs: negative 1} and \ref{obs: negative 3}, types 5.3, 6.4, 6.5 and 6.9 are excluded when choosing the first Kokotsakis quadrilateral.

\subsection{Other special solutions} \label{subsection: other}

It is natural to consider whether other operations are possible after constructing a row of rigid-foldable Kokotsakis quadrilaterals. For instance, the new rows can be generated by substituting some sector angles by their complements to $\pi$, or switching the position of some sector angles. From our observation, generically such operations

\begin{enumerate} [(1)]
	\item do not result in another rigid-foldable Kokotsakis quadrilateral.
	\item need two or more extra conditions for ensuring the sector angles on each inner panel to be $2\pi$, where the total number of constraints exceeds 8.
\end{enumerate}

Therefore only repeating is considered in method 1. We do, however, admit that repeating is only an elementary special solution, and possibly these operations could be applied in some non-generic cases.

Further, there are some more special solutions which are more dependent on symmetry. \cite{dieleman_jigsaw_2020} provided the \textit{tiling type} in (Section \ref{subsection: tiling}). However, it is unclear to us how to use ``tiling'' to generate a large quadrilateral creased paper from a solved row of rigid-foldable Kokotsakis quadrilaterals. Some insights on finding more types of rigid-foldable quadrilateral creased papers are provided in Section \ref{section: discussion}.

\section{Discussion} \label{section: discussion}

\subsection{Comment on further solutions}

Here we point out some considerations for further exploration of large rigid-foldable quadrilateral creased papers.

\begin{enumerate} [(1)]
	\item Although for a large quadrilateral creased paper, generically the number of constraints for rigid-foldability is much greater than the number of sector angles, it maybe still possible to solve this system as a whole and find some special solutions.
	\item Following the idea of solving the sector angles quadrilateral by quadrilateral, in Section \ref{section: reason} we move forward only by comparing the number of constraints and variables. However, this is not strict for non-linear systems. By studying the rigid-foldability conditions more carefully there may be further special solutions. Then it might be possible to go beyond the known types and find more general rigid-foldable quadrilateral creased papers.  
\end{enumerate}

\subsection{Comment on the numerical results}

For the types mentioned above, only the orthodiagonal, isogonal, parallel repeating, and tiling do not require solving complicated equations, while other types need some input sector angles and solve the rest of the sector angles numerically. We will comment on these numerical solutions here. 

\begin{enumerate}[(1)]
	\item  When making examples we used \textit{fsolve} in MATLAB with random input sector angles and initial values for solution. The numerical solutions can be sensitive to the input sector angles and initial values. Given a set of input sector angles there might be no solution, or only complex solutions, as mentioned in Observation \ref{obs: negative 1}. Sometimes an exact numerical solution can be obtained only after trying many different input sector angles and initial values. Changing the position of input sector angles will also affect the numerical result. 
	
	\item Most of equations in the rigid-foldability condition are trigonometric, but some are exponential or elliptic. Trigonometric equations can be transformed into polynomial equations with the following relations:
	\begin{equation}
	\cos \alpha = \dfrac{1-\tan^2 \dfrac{\alpha}{2}}{1+\tan^2 \dfrac{\alpha}{2}}, \quad \sin \alpha = \dfrac{2\tan \dfrac{\alpha}{2}}{1+\tan^2 \dfrac{\alpha}{2}}
	\end{equation}
	which has the advantage that the numerical solution of a system of polynomial equations has been well studied. It would be possible to apply more advanced numerical methods here for better results.
	
	\item In Observation \ref{obs: negative 3}, we mentioned that types 6.4 and 6.5 can be solved, but are rigid in $\mathbb{R}^3$. This is because the rigid-foldability constraints on sector angles are derived in the complexified configuration space, where the real folding angles may be isolated. 
\end{enumerate}

Another phenomenon in the numerical solution is called \textit{degeneration}, which can be divided into two types. 

\begin{enumerate}[(a)]
	\item An (anti)deltoid or (anti)isogram may degenerate to a cross (two pairs of collinear creases). A conic vertex may degenerate to an (anti)deltoid or (anti)isogram. An elliptic vertex may degenerate to a conic vertex. Further degeneration is also possible.
	\item There might be unexpected relations among the sector angles, such as in figure \ref{fig: linear repeating}(e) and \ref{fig: linear repeating}(f), some inner creases in different linear units are parallel; and in figure \ref{fig: hybrid repeating}(a) and \ref{fig: hybrid repeating}(b) when solving new sector angles in a row some values recur.
\end{enumerate}

Apart from a vertex becoming a cross, such degenerations are acceptable.

\subsection{Choice of crease lengths and possible self-intersection}

After finding all the sector angles of a large quadrilateral creased paper, the length of inner creases in a row and a column is adjustable. When plotting the creased paper based on a set of known sector angles, the length of inner creases should be adjusted to avoid intersection of inner creases at some points other than vertices.

Further, even if the creased paper is successfully constructed, there might be self-intersection in the rigid folding motion. Apart from simulation, efficient methods to predict self-intersection are still unknown. We provide a discussion in \cite{he_rigid_2019}.

\subsection{Perturbation Method}

If consider finding rigid-foldable quadrilateral creased paper numerically, a possible method is applying perturbation to Miura-ori (in $\mathbb{R}^3$) and examining the error step by step. If the series converges, the limit would be a new solution. This numeric method might help us to find new variations that cannot be explained by the analytical solutions mentioned in this article. However, the algorithm must be carefully designed to ensure convergence and be capable of finding new solutions.

\section{Conclusion}

Based on a nearly complete classification of rigid-foldable Kokotsakis quadrilaterals from Ivan Izmestiev, this paper describes several new variations of large rigid-foldable quadrilateral creased papers, without any restriction on the developability or flat-foldability. The rigid folding motions of these new variations are more irregular but still have single degree of freedom.

\vskip6pt

\enlargethispage{20pt}

\begin{description}
	\item [Data Access] This article has a supplementary material.
	\item [Author Contribute] ZH developed the theories described here, and took the primary role in writing the paper. SDG provided advice on, and made contributions to, the development of the project; and helped to write the paper. All authors gave final approval for publication.
	\item [Competing] We have no competing interests.
	\item [Funding] ZH was partially funded by a studentship awarded by the George and Lilian Schiff Foundation.
	\item [Acknowledge] We thank Ivan Izmestiev and Grigory Ivanov for helpful discussions at the workshop organized by the Erwin Schrodinger Institute on "Rigidity and Flexibility of Geometric Structures". We also thank Martin Van Hecke for helpful discussions at the 7th International Meeting on Origami, Science, Mathematics, and Education.
\end{description}


\vskip2pt

\bibliographystyle{unsrt}

\bibliography{Rigid-Folding}

\end{document}


\maketitle
	
\numberwithin{equation}{section}
\numberwithin{prop}{section}

In this supplementary material we will provide the configuration space of a degree-4 single-vertex creased paper in Section \ref{section: s1}, then present all types of rigid-foldable Kokotsakis quadrilaterals in Section \ref{section: s2}. These two topics are the base of the classification of large rigid-foldable quadrilateral mesh in the main text. On the way to a full solution of the two topics, \cite{bricard_memoire_1897, sauer_uber_1931} presented some primary results at around a hundred years ago, \cite{stachel_kinematic_2010, nawratil_reducible_2011, nawratil_reducible_2012} made some progress, and \cite{izmestiev_classification_2016} gave the full characterization. Here we present the result from the comprehensive catalogue given in \cite{izmestiev_classification_2016}, which is adapted to make it more intelligible for the result provided in the main text.

\section{Degree-4 single-vertex creased paper} \label{section: s1}

In this section we will describe the relation among adjacent folding angles in a degree-4 single-vertex creased paper. As figure \ref{fig: degree-4 vertex} shows, the sector angles $\alpha, \beta, \gamma, \delta$ and tangent of half of a pair of adjacent folding angles $z, w$ are labelled. Here $\alpha, \beta, \gamma, \delta \in (0, \pi)$, and $z,w \in [-\pi, \pi]$. In order to present the results clearly, we classify a single-vertex creased paper based on the following principle.

\begin{defn}
	Consider the following equation
	\begin{equation} \label{eq: single-vertex characterization}
	\alpha \pm \beta \pm \gamma \pm \delta=0~(\mathrm{mod}~2\pi)
	\end{equation}
	A degree-4 single-vertex creased paper is said to be
	\begin{enumerate} [(1)]
		\item an \textit{isogram}, if $\alpha=\gamma, \beta=\delta$, and an \textit{antiisogram}, if $\alpha+\gamma=\pi, \beta+\delta=\pi$.
		\item a \textit{deltoid}, if $\alpha=\beta, \delta=\gamma$ or $\alpha=\gamma, \beta=\delta$, and an \textit{antideltoid}, if $\alpha+\beta=\pi, \delta+\gamma=\pi$ or $\alpha+\gamma=\pi, \beta+\delta=\pi$.
		\item of \textit{conic} type, if equation \eqref{eq: single-vertex characterization} has exactly one solution;
		\item of \textit{elliptic} type, if equation \eqref{eq: single-vertex characterization} has no solutions;
	\end{enumerate}
\end{defn}

For different types, the forms of expression of $z$ and $w$ (possibly with respect to a parameter $t$) are different, which are listed below.

\begin{prop}
	For an (anti)isogram, the relation of $z$ and $w$ is:
	\begin{enumerate}[(a)]
		\item if $\alpha=\beta=\gamma=\delta=\pi/2$,
		\begin{equation}
		z=0 \quad \mathrm{or} \quad z=\infty \quad \mathrm{or} \quad w=0 \quad \mathrm{or} \quad w=\infty
		\end{equation}
		\item if the single creased paper is an antiisogram and at the same time an antideltoid, 
		\begin{equation}
		\left\{
		\begin{array}{lcl}  
		z=0~~\mathrm{or}~~z=\frac{w}{\cos \alpha}~&\mathrm{if}~\alpha=\beta=\pi-\gamma=\pi-\delta  \\  
		w=0~~\mathrm{or}~~w=-\frac{z}{\cos \alpha}~&\mathrm{if}~\alpha=\delta=\pi-\beta=\pi-\gamma   
		\end{array}  
		\right.  
		\end{equation}
		\item if the single creased paper is an antiisogram but not an antideltoid,
		\begin{equation}
		z=\kappa w,~\mathrm{where}~\kappa=\dfrac{\sin \frac{\alpha-\beta}{2}}{\sin \frac{\alpha+\beta}{2}}~~\mathrm{or}~~\dfrac{\cos \frac{\alpha-\beta}{2}}{\cos \frac{\alpha+\beta}{2}}
		\end{equation}
		\item if the single creased paper is an isogram and at the same time a deltoid,
		\begin{equation}
		z=\infty~~\mathrm{or}~~
		z=
		\dfrac{\cos \alpha}{w}
		\end{equation}
		\item if the single creased paper is an isogram but not a deltoid,
		\begin{equation}
		z=\dfrac{1}{\kappa w},~\mathrm{where}~\kappa=\dfrac{\sin \frac{\alpha-\beta}{2}}{\sin \frac{\alpha+\beta}{2}}~~\mathrm{or}~~\dfrac{\cos \frac{\alpha-\beta}{2}}{\cos \frac{\alpha+\beta}{2}}
		\end{equation}
		\end{enumerate}
\end{prop}

In the above cases all the variables and parameters are real numbers. However, in the next propositions, $z$ and $w$ are real numbers with respect to a complex parameter $t$.

\begin{figure} [!tb]
	\noindent \begin{centering}
		\includegraphics[width=0.4\linewidth]{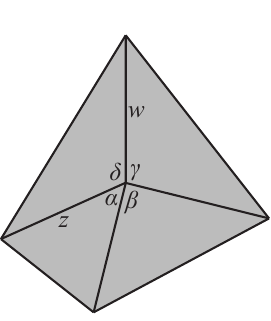}
		\par \end{centering}
	
	\caption{\label{fig: degree-4 vertex}(a) A degree-4 single-vertex creased paper. We label these sector angles as $\alpha, \beta, \gamma, \delta$ counterclockwisely. $z$ and $w$ are the tangent of half of folding angles on the corresponding inner creases. $z$ is between sector angles $\delta$ and $\alpha$, $w$ is between sector angles $\gamma$ and $\delta$. $z=\tan \dfrac{\rho_{\delta \alpha}}{2}$, $w=\tan \dfrac{\rho_{\alpha \beta}}{2}$.}
\end{figure}

\begin{prop}
	For an (anti)deltold but not an (anti)isogram, the expression of $z$ and $w$ (possibly with respected to a parameter $t$) is (for square roots, both the real and imaginary parts are positive, $i=\sqrt{-1}$):
\begin{enumerate}[(a)]
	\item if $\alpha=\delta$ and $\beta=\gamma$
	\begin{equation}
	\begin{gathered}
	z=\infty\\
	\mathrm{or} \\
	z=p \sin t, \quad w=\epsilon \sqrt{-\mu}e^{it}
	\end{gathered}
	\end{equation}
	where
	\begin{equation*}
	\begin{gathered}
	p=\sqrt{\dfrac{\sin^2 \delta}{\sin^2 \gamma}-1}, \quad
	\epsilon=\left\{	\begin{array}{lcl}  
	1 &\mathrm{if}~\gamma+\delta>\pi  \\  
	-1 &\mathrm{if}~\gamma+\delta<\pi
	\end{array},  
	\right. ~~~
	\mu=\dfrac{\tan \delta+\tan \gamma}{\tan \delta-\tan \gamma} 
	\end{gathered} 
	\end{equation*}
	\item if $\alpha+\delta=\beta+\gamma=\pi$,
	\begin{equation}
	\begin{gathered}
	z=0\\
	\mathrm{or} \\
	z^{-1}=p \sin t, \quad w=\epsilon \sqrt{-\mu}e^{it}
	\end{gathered}
	\end{equation}
	where
	\begin{equation*}
	\begin{gathered}
	p=\sqrt{\dfrac{\sin^2 \delta}{\sin^2 \gamma}-1}, \quad \epsilon=\left\{	\begin{array}{lcl}  
	-1 &\mathrm{if}~\gamma+\delta>\pi  \\  
	1 &\mathrm{if}~\gamma+\delta<\pi
	\end{array},  
	\right. ~~~ \mu=\dfrac{\tan \delta+\tan \gamma}{\tan \delta-\tan \gamma} 
	\end{gathered} 
	\end{equation*}
	\item if $\alpha=\beta$ and $\gamma=\delta$,
	\begin{equation}
	\begin{gathered}
	w=\infty\\
	\mathrm{or}\\
	\quad z=\iota \sqrt{-\lambda}e^{it}, \quad w=q \sin t
	\end{gathered}
	\end{equation}
	where
	\begin{equation*}
	\begin{gathered}
	q=\sqrt{\dfrac{\sin^2 \delta}{\sin^2 \alpha}-1}, \quad \iota=\left\{
	\begin{array}{lcl}  
	1 &\mathrm{if}~\alpha+\delta>\pi  \\  
	-1 &\mathrm{if}~\alpha+\delta<\pi
	\end{array},  
	\right. ~~~ \lambda=\dfrac{\tan \delta+\tan \alpha}{\tan \delta-\tan \alpha} 
	\end{gathered} 
	\end{equation*}
	\item if $\alpha+\beta=\gamma+\delta=\pi$,
	\begin{equation}
	\begin{gathered}
	w=0 \\
	\mathrm{or} \\
	\quad z=\iota \sqrt{-\lambda}e^{it}, \quad w^{-1}=q \sin t
	\end{gathered}
	\end{equation}
	where
	\begin{equation*}
	\begin{gathered}
	q=\sqrt{\dfrac{\sin^2 \delta}{\sin^2 \alpha}-1}, \quad \iota=\left\{
	\begin{array}{lcl}  
	-1 &\mathrm{if}~\alpha+\delta>\pi  \\  
	1 &\mathrm{if}~\alpha+\delta<\pi
	\end{array}, 
	\right. ~~~ \lambda=\dfrac{\tan \delta+\tan \alpha}{\tan \delta-\tan \alpha}
	\end{gathered} 
	\end{equation*}
	\end{enumerate}
\end{prop}

In this case, one branch is trivial and another is non-trivial. The domain of $t$ for the non-trivial branch is $\pm \dfrac{\pi}{2}+i\mathbb{R}$ if $\mu$ or $\lambda$ is positive; and is $\mathbb{R} \cup (\pi+i\mathbb{R})$ if $\mu$ or $\lambda$ is negative.

\begin{prop} \label{prop: conic}
	For the conic type, the expression of $z$ and $w$ with respected to a parameter $t$ is (for square roots, both the real and imaginary parts are positive):
	\begin{equation}
	z^m=p \sin t, \quad w^n=q \sin(t+t_0)
	\end{equation}
	where
	\begin{center}
	\begin{tabular}{|c|c|c|}
		\hline  
		&$n=1$&$n=-1$\\
		\hline  
		$m=1$&$\alpha+\gamma=\beta+\delta$&$\alpha+\beta=\gamma+\delta$\\
		\hline 
		$m=-1$&$\alpha+\delta=\beta+\gamma$&$\alpha+\beta+\gamma+\delta=2\pi$\\
		\hline
	\end{tabular}
	\end{center}
	\begin{equation*}
	\begin{gathered}
	p=\sqrt{\dfrac{\sin \alpha \sin \delta}{\sin \beta \sin \gamma}-1}, \quad q=\sqrt{\dfrac{\sin \gamma \sin \delta}{\sin \alpha \sin \beta}-1} \\
	t_0=\tan^{-1}\left(i\sqrt{\dfrac{\sin \beta \sin \delta}{\sin \alpha \sin \gamma}}\right)
	\end{gathered} 
	\end{equation*}
	when deriving $t_0$, the phase choice is
	\begin{center}
	\begin{tabular}{|c|c|c|c|}
		\hline  
		$t_0 \in$ &$pq \in \mathbb{R}_+$&$pq \in i\mathbb{R}_+$&$pq \in \mathbb{R}_-$\\
		\hline  
		$\sigma<\pi$& $i\mathbb{R}_+$ &$\frac{\pi}{2}+i\mathbb{R}_+$&$\pi+i\mathbb{R}_+$\\
		\hline 
		$\sigma>\pi$&$\pi+i\mathbb{R}_+$&$\frac{3\pi}{2}+i\mathbb{R}_+$&$i\mathbb{R}_+$\\
		\hline
	\end{tabular}
	\end{center}
	and
	\begin{equation*}
	\sigma=\left\{
	\begin{array}{lcl}  
	\frac{\alpha+\beta+\gamma+\delta}{2} &\mathrm{if}~\alpha+\gamma=\beta+\delta\\  
	\frac{-\alpha+\beta+\gamma-\delta}{2}+\pi &\mathrm{if}~\alpha+\beta=\gamma+\delta \\
	\frac{\alpha+\beta-\gamma-\delta}{2}+\pi &\mathrm{if}~\alpha+\delta=\beta+\gamma
	\\
	\frac{-\alpha+\beta-\gamma+\delta}{2}+\pi &\mathrm{if}~\alpha+\beta+\delta+\gamma=2\pi
	\end{array}  
	\right.
	\end{equation*}
\end{prop}

\begin{prop} \label{prop: elliptic}
	For the elliptic type, the expression of $z$ and $w$ with respected to a parameter $t$ is determined by (for square roots, both the real and imaginary parts are positive):
	\begin{equation}
	\begin{gathered}
	\sigma=\dfrac{\alpha+\beta+\gamma+\delta}{2}, \quad M=\dfrac{\sin \alpha \sin \beta \sin \gamma \sin \delta}{\sin (\sigma-\alpha) \sin (\sigma-\beta) \sin (\sigma-\gamma) \sin (\sigma-\delta)}
	\end{gathered}
	\end{equation}
	for the elliptic type $M \neq 1$, and
	\begin{equation*}
	\begin{gathered}
	M<1 \Leftrightarrow \max(\alpha,\beta,\gamma,\delta)+\min(\alpha,\beta,\gamma,\delta)<\sigma<\pi,~~\mathrm{or}~~\\\max(\alpha,\beta,\gamma,\delta)+\min(\alpha,\beta,\gamma,\delta)>\sigma>\pi \\
	M>1 \Leftrightarrow \max(\alpha,\beta,\gamma,\delta)+\min(\alpha,\beta,\gamma,\delta)>\sigma<\pi,~~\mathrm{or}~~\\\max(\alpha,\beta,\gamma,\delta)+\min(\alpha,\beta,\gamma,\delta)<\sigma>\pi
	\end{gathered}
	\end{equation*}
	\begin{enumerate} [(a)]
		\item if $M<1$, 
		\begin{equation}
			z=p\mathrm{~sn~}t, \quad w=q\mathrm{~sn}(t+t_0)
		\end{equation}
	where $\mathrm{sn}$ is the Jacobian elliptic sine function with modulus $k=\sqrt{1-M}$, and 
	\begin{equation*}
		t_0=\mathrm{dn}^{-1}\left(\sqrt{\dfrac{\sin \alpha \sin \gamma}{\sin (\sigma-\alpha) \sin (\sigma-\gamma)}}\right)
	\end{equation*}
	where $\mathrm{dn}$ is the Jacobian elliptic tangent function with modulus $k=\sqrt{1-M}$.
	\item if $M>1$, 
	\begin{equation}
	z=p\mathrm{~cn~}t, \quad w=q\mathrm{~cn}(t+t_0)
	\end{equation}
	where $\mathrm{cn}$ is the Jacobian elliptic cosine function with modulus $k=\sqrt{1-1/M}$, and 
	\begin{equation*}
		t_0=\mathrm{dn}^{-1}\left(\sqrt{\dfrac{\sin (\sigma-\alpha) \sin (\sigma-\gamma)}{\sin \alpha \sin \gamma}}\right)
	\end{equation*}
	where $\mathrm{dn}$ is the Jacobian elliptic tangent function with modulus $k=\sqrt{1-1/M}$.
	\end{enumerate}
In both cases, the other parameters are
	\begin{equation*}
	\begin{gathered}
	p=\sqrt{\dfrac{\sin \alpha \sin \delta}{\sin (\sigma-\alpha) \sin (\sigma-\delta)}-1}, \quad q=\sqrt{\dfrac{\sin \gamma \sin \delta}{\sin (\sigma-\gamma) \sin (\sigma-\delta)}-1} \\
	\end{gathered} 
	\end{equation*}
	when deriving $t_0$, the phase choice is
	\begin{center}
		\begin{tabular}{|c|c|c|c|}
			\hline  
			$t_0 \in$ &$pq \in \mathbb{R}_+$&$pq \in i\mathbb{R}_+$&$pq \in \mathbb{R}_-$\\
			\hline  
			$\sigma<\pi$& $(0, iK')$ &$(K, K+iK')$&$(2K,2K+iK')$\\
			\hline 
			$\sigma>\pi$&$(2K,2K+iK')$&$(3K,3K+iK')$&$(0, iK')$\\
			\hline
		\end{tabular}
	\end{center}
Here $K$ and $iK'$ are the quarter periods of the elliptic function determined by the modulus $k$.
	\begin{equation*}
	K=\int_{0}^{\frac{\pi}{2}}\dfrac{\dif \theta}{\sqrt{1-k^2\sin \theta}}, \quad K'=\int_{0}^{\frac{\pi}{2}}\dfrac{\dif \theta}{\sqrt{1-(1-k^2)\sin \theta}}
	\end{equation*}
\end{prop}

For further discussion, we point out a special case of a degree-4 single creased paper.

\begin{prop}
	If the sector angles satisfy:
	\begin{equation}
	\cos \alpha \cos \gamma=\cos \beta \cos \delta
	\end{equation}
	the creased paper is named to be \textit{orthodiagonal} type, which is a special elliptic type but may also degenerate to (anti)deltoids. To finish setting up notation some parameters are listed for this type.
	\begin{equation*}
	\begin{gathered}
	\lambda=\left\{
	\begin{array}{lcl}  
	\frac{\tan \delta+\tan \alpha}{\tan \delta-\tan \alpha} &\mathrm{if}~\alpha \neq \frac{\pi}{2}~~\mathrm{or}~~\delta \neq \frac{\pi}{2}  \\  
	\frac{\cos \beta+\cos \gamma}{\cos \beta-\cos \gamma} &\mathrm{if}~\alpha=\delta=\frac{\pi}{2}
	\end{array}  
	\right.
	\\
	\mu=\left\{
	\begin{array}{lcl}  
	\frac{\tan \delta+\tan \gamma}{\tan \delta-\tan \gamma} &\mathrm{if}~\gamma \neq \frac{\pi}{2}~~\mathrm{or}~~\delta \neq \frac{\pi}{2}  \\  
	\frac{\cos \beta+\cos \alpha}{\cos \beta-\cos \alpha} &\mathrm{if}~\gamma=\delta=\frac{\pi}{2}
	\end{array}  
	\right. \\
	\nu=\left\{
	\begin{array}{lcl}  
	\frac{(\lambda-1)(\mu-1)}{\cos \delta} &\mathrm{if}~\delta \neq \frac{\pi}{2} \\  
	2(\mu-1)\tan \alpha &\mathrm{if}~\delta=\gamma=\frac{\pi}{2} \\
	2(\lambda-1)\tan \gamma &\mathrm{if}~\delta=\alpha=\frac{\pi}{2}
	\end{array}  
	\right.
	\end{gathered}
	\end{equation*}
	If $\alpha=\delta$, $\beta=\gamma$, or $\alpha+\delta=\beta+\gamma=\pi$, 
	\begin{equation*}
	\begin{gathered}
	\zeta=\left\{
	\begin{array}{lcl}  
	\frac{\mu-1}{\cos \delta} &\mathrm{if}~\delta \neq \frac{\pi}{2} \\  
	2\tan \gamma &\mathrm{if}~\delta=\alpha=\frac{\pi}{2}
	\end{array}  
	\right.
	\end{gathered}
	\end{equation*}
If $\alpha=\beta$, $\gamma=\delta$, or $\alpha+\beta=\gamma+\delta=\pi$,
	\begin{equation*}
	\begin{gathered}
	\xi=\left\{
	\begin{array}{lcl}  
	\frac{\lambda-1}{\cos \delta} &\mathrm{if}~\delta \neq \frac{\pi}{2} \\  
	2\tan \alpha &\mathrm{if}~\delta=\gamma=\frac{\pi}{2}
	\end{array}  
	\right.
	\end{gathered}
\end{equation*}
\end{prop}

\section{The list of rigid-foldable Kokotsakis quadrilaterals} \label{section: s2}

There are in total 7 categories of rigid-foldable Kokotsakis quadrilaterals, respectively named based on the types of degree-4 single-vertex creased papers involved and the relation among the tangent of half of their folding angles. All the sector angles in a Kokotsakis quadrilateral are labelled $\alpha_j, \beta_j, \gamma_j, \delta_j(1 \le j \le 4)$, as shown in figure \ref{fig: kokotsakis}(a). Note that there are three concerns. First, there might be no solution for sector angles in the real field. Second, even if there is a solution in the real field, the rigid folding motion may locates in the complexified configuration space, hence the Kokotsakis quadrilateral will still be rigid in $\mathbb{R}^3$. Third, even if there is a solution in the real field and there is a rigid folding motion in $\mathbb{R}^3$, the panels may constantly self-intersect, which makes the rigid folding motion not valid. 

Specifically these categories are:

\begin{enumerate} [1.]
	\item Orthodiagonal
	\item Equimodular
	\item Conjugate Modular
	\item Linear Compounds
	\item Linearly Conjugate
	\item Chimeras
	\item Trivial
\end{enumerate}

\begin{figure} [!tb]
	\noindent \begin{centering}
		\includegraphics[width=1\linewidth]{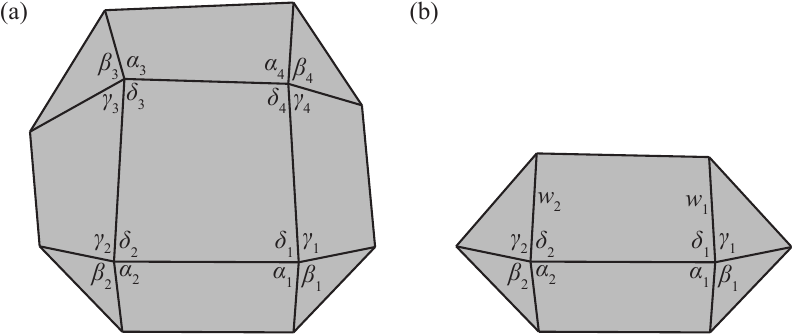}
		\par \end{centering}
	
	\caption{\label{fig: kokotsakis}(a) A Kokotsakis quadrilateral (not necessarily developable). We label these sector angles as $\alpha_j, \beta_j, \gamma_j, \delta_j~(1 \le j \le 4)$. (b) A linear unit, which is half of a Kokotsakis quadrilateral with $w_1=cw_2$, $c$ is a constant. If two linear units have the same $c$, they could be stitched together and form a rigid-foldable Kokotsakis quadrilateral.}
\end{figure}

In all types above we first have

\begin{equation}
\delta_1+\delta_2+\delta_3+\delta_4=2\pi
\end{equation}

Note that all parameters used below are defined for each type of single creased paper in Section \ref{section: s1}.

\begin{enumerate}[(a)]
	\item for an (anti)isogram: $\kappa$
	\item for an (anti)deltoid: $p,q,\epsilon,\iota,\mu,\lambda$
	\item for the conic type: $p,q,t_0,\sigma$. In the following equations if referring to a conic type and $t_j~(1 \le j \le 4)$ appears, each of them means $t_0$ in proposition \ref{prop: conic}.
	\item for the elliptic type: $\sigma, M, k, p, q, t_0, K, K'$. In the following equations if referring to a elliptic type and $t_j~(1 \le j \le 4)$ appears, each of them means $t_0$ in proposition \ref{prop: elliptic}.
	\item for the orthodiagonal type, $\lambda, \mu, \nu, \zeta, \xi$.
\end{enumerate}

\begin{description}
	\item[Type 1.1 Orthodiagonal Involutive]
	\begin{equation}
	\begin{gathered}
	\delta_j ~(1 \le j \le 4) \neq \dfrac{\pi}{2} \\
	\left\{
	\begin{array}{ccl}
	\alpha_1+\alpha_2=\alpha_3+\alpha_4=\delta_1+\delta_2=\pi, \quad \dfrac{\tan{\gamma_1}}{\tan{\gamma_4}}=\dfrac{\tan{\gamma_2}}{\tan{\gamma_3}}=\dfrac{\tan{\delta_1}}{\tan{\delta_4}} \\
	\mathrm{or} \\
	\gamma_1+\gamma_4=\gamma_2+\gamma_3=\delta_1+\delta_4=\pi, \quad 
	\dfrac{\tan{\alpha_1}}{\tan{\alpha_2}}=\dfrac{\tan{\alpha_4}}{\tan{\alpha_3}}=\dfrac{\tan{\delta_1}}{\tan{\delta_2}}
	\end{array}  
	\right. \\
	\cos \alpha_j \cos \gamma_j=\cos \beta_j \cos \delta_j ~(1 \le j \le 4)
	\end{gathered}
	\end{equation}
	\item[Type 1.2 Orthodiagonal Antiinvolutive]
	\begin{equation}
	\begin{gathered}
	\alpha_j \pm \beta_j \pm \gamma_j \pm \delta_j \neq 0~(\mathrm{mod}~2\pi, 1 \le j \le 4) \\
	\cos \alpha_j \cos \gamma_j=\cos \beta_j \cos \delta_j ~(1 \le j \le 4) \\
	\lambda_1=-\lambda_2, \quad \lambda_3=-\lambda_4, \quad \mu_1=-\mu_4, \quad \mu_2=-\mu_3 \\
	\dfrac{\nu_1^2}{\lambda_1\mu_1}=\dfrac{\nu_3^2}{\lambda_3\mu_3}, \quad \dfrac{\nu_2^2}{\lambda_2\mu_2}=\dfrac{\nu_3^4}{\lambda_4\mu_4}, \quad \dfrac{\nu_1^2}{\lambda_1\mu_1}+\dfrac{\nu_2^2}{\lambda_2\mu_2}=1
	\end{gathered}
	\end{equation}
	\item[Type 2.1 Elliptic Equimodular]
	\begin{equation}
	\begin{gathered}
	\alpha_j \pm \beta_j \pm \gamma_j \pm \delta_j \neq 0~(\mathrm{mod}~2\pi, 1 \le j \le 4) \\
	M_1=M_2=M_3=M_4 \\
	p_1=p_2, \quad p_3=p_4, \quad q_1=q_4, \quad q_2=q_3 \\
	t_1 \pm t_2 \pm t_3 \pm t_4 \in \left\{
	\begin{array}{ccl}  
	4K\mathbb{Z}+2iK'\mathbb{Z} & \mathrm{if}~ M_1<1 \\  
	4K\mathbb{Z}+2(K+iK')\mathbb{Z} & \mathrm{if}~M_1>1
	\end{array}  
	\right. \\
	\end{gathered}
	\end{equation}
	Any combination of signs is sufficient for the last condition.
	\item[Type 2.2 Conic Equimodular]
	\begin{equation}
	\begin{gathered}
	\alpha_j + \gamma_j = \beta_j + \delta_j~(1 \le j \le 4) \\
	p_1=p_2, \quad p_3=p_4, \quad q_1=q_4, \quad q_2=q_3 \\
	t_1 \pm t_2 \pm t_3 \pm t_4 \in 2\pi \mathbb{Z}
	\end{gathered}
	\end{equation}
	Any combination of signs is sufficient for the last condition.
	\item[Type 3.1 First Elliptic Conjugative Modular]
	\begin{equation}
	\begin{gathered}
	M_j>1 ~(1 \le j \le 4) \\
	M_1=M_3, \quad M_2=M_4, \quad  \dfrac{1}{M_1}+\dfrac{1}{M_2}=1 \\
	\dfrac{p_1}{p_2}=\pm i\sqrt{M_1-1}, \quad \dfrac{p_3}{p_4}=\pm i\sqrt{M_1-1} \\ \dfrac{q_3}{q_2}=\pm i\sqrt{M_1-1}, \quad
	\dfrac{q_1}{q_4}=\pm i\sqrt{M_1-1} \\
	t_1\pm it_2 \pm t_3 \pm it_4=\left\{
	\begin{array}{ccl}  
	4K\mathbb{Z}+2(K+iK')\mathbb{Z} &
	\begin{array}{ccl}
	\mathrm{if}~\frac{p_1}{p_2}=\frac{q_3}{q_2},~\frac{p_3}{p_4}=\frac{q_3}{q_4} \\ \mathrm{or}~~\frac{p_1}{p_2}=-\frac{q_3}{q_2},~\frac{p_3}{p_4}=-\frac{q_3}{q_4}
	\end{array} \\
	2K+4K\mathbb{Z}+2(K+iK')\mathbb{Z} &
	\begin{array}{ccl}
	\mathrm{if}~\frac{p_1}{p_2}=\frac{q_3}{q_2},~\frac{p_3}{p_4}=-\frac{q_3}{q_4} \\ \mathrm{or}~~\frac{p_1}{p_2}=-\frac{q_3}{q_2},~\frac{p_3}{p_4}=\frac{q_3}{q_4}
	\end{array}
	\end{array}  
	\right.
	\end{gathered}
	\end{equation}
	Any combination of signs is sufficient for the last condition, which also implies $t_1\pm t_3 \in (K+iK')\mathbb{Z}$ and $t_2 \pm t_4 \in (K+iK')\mathbb{Z}$.
	\item[Type 3.2 Second Elliptic Conjugative Modular]
	\begin{equation}
	\begin{gathered}
	M_j>1 ~(1 \le j \le 4) \\
	M_1=M_4, \quad M_2=M_3, \quad  \dfrac{1}{M_1}+\dfrac{1}{M_2}=1 \\
	\dfrac{p_1}{p_2}=\pm i\sqrt{M_1-1}, \quad \dfrac{p_4}{p_3}=\pm i\sqrt{M_1-1} \\ q_2=q_3, \quad q_1=q_4 \\
	t_1 \pm it_2 \pm t_3 \pm it_4=\left\{
	\begin{array}{ccl}  
	4K\mathbb{Z}+2(K+iK')\mathbb{Z} &
	\mathrm{if}~\frac{p_1}{p_2}=\frac{p_4}{p_3}\\
	2K+4K\mathbb{Z}+2(K+iK')\mathbb{Z} & \mathrm{if}~\frac{p_1}{p_2}=-\frac{p_4}{p_3}
	\end{array}  
	\right.
	\end{gathered}
	\end{equation}
	Any combination of signs is sufficient for the last condition, which also implies $t_1+t_4 \in (K+iK')\mathbb{Z}$ and $t_2+t_3 \in (K+iK')\mathbb{Z}$.
	\item[Type 4 Linear Compounds]
	This type of rigid-foldable Kokotsakis quadrilateral is formed by stitching of two \textit{linear units}, each of which is half of a Kokotsakis quadrilateral and is graphically explained in figure \ref{fig: kokotsakis}. Since $w_1=cw_2$, if two linear units have the same $c$, they could be stitched together and form a rigid-foldable Kokotsakis quadrilateral. Below we list the 5 types of linear units.
	\begin{description}
		\item [Type 4.1 (Anti)isogram]
		\begin{equation}
		\begin{gathered}
		\left\{
		\begin{array}{ccl} 
		\alpha_1=\gamma_1, \quad \beta_1=\delta_1, \quad \alpha_2=\gamma_2, \quad \beta_2=\delta_2 & \mathrm{(S2.8a)}\\
		\mathrm{or} \\ 
		\alpha_1+\gamma_1=\beta_1+\delta_1=\pi, \quad \alpha_2+\gamma_2=\beta_2+\delta_2=\pi & \mathrm{(S2.8b)}
		\end{array}
		\right. \\
		c=\dfrac{\kappa_2}{\kappa_1}
		\end{gathered}
		\end{equation}
		\item [Type 4.2 Lateral coupled (anti)deltoid]
		\begin{equation}
		\begin{gathered}
		\left\{
		\begin{array}{ccl} 
		\alpha_1=\beta_1, \quad \gamma_1=\delta_1, \quad \alpha_2=\beta_2, \quad \gamma_2=\delta_2 & \mathrm{(S2.9a)} \\
		\mathrm{or} \\ 
		\alpha_1+\beta_1=\gamma_1+\delta_1=\pi, \quad \alpha_2+\beta_2=\gamma_2+\delta_2=\pi & \mathrm{(S2.9b)}  
		\end{array}
		\right. \\
		\lambda_1=\lambda_2 \\
		c=\left\{
		\begin{array}{ccl} 
		\frac{\xi_2}{\xi_1} & \mathrm{if}~\mathrm{(S2.9a)} \\
		\frac{\xi_1}{\xi_2} & \mathrm{if}~\mathrm{(S2.9b)}
		\end{array}
		\right. \\
		\end{gathered}
		\end{equation}
		\item [Type 4.3 Frontal coupled (anti)deltoid]
		\begin{equation}
		\begin{gathered}
		\left\{
		\begin{array}{ccl} 
		\alpha_1=\delta_1, \quad \beta_1=\gamma_1, \quad \alpha_2=\delta_2, \quad \beta_2=\gamma_2 & \mathrm{(S2.10a)} \\
		\mathrm{or} \\ 
		\alpha_1+\delta_1=\beta_1+\gamma_1=\pi, \quad \alpha_2+\delta_2=\beta_2+\gamma_2=\pi & \mathrm{(S2.10b)}
		\end{array}
		\right. \\
		p_1=p_2 \\
		c=\dfrac{\epsilon_1}{\epsilon_2}\sqrt{\dfrac{\mu_1}{\mu_2}}
		\end{gathered}
		\end{equation}
		\item [Type 4.4 Conic]
		\begin{equation}
		\begin{gathered}
		\left\{
		\begin{array}{ccl} 
		\alpha_1+\gamma_1=\beta_1+\delta_1, \quad \alpha_2+\gamma_2=\beta_2+\delta_2 & \mathrm{(S2.11a)} \\
		\mathrm{or} \\ 
		\alpha_1+\beta_1+\gamma_1+\delta_1=2\pi, \quad \alpha_2+\beta_2+\gamma_2+\delta_2=2\pi & \mathrm{(S2.11b)}
		\end{array}
		\right. \\
		\dfrac{\sin \alpha_1}{\sin \beta_1}=\dfrac{\sin \alpha_2}{\sin \beta_2}, \quad \dfrac{\sin \gamma_1}{\sin \delta_1}=\dfrac{\sin \gamma_2}{\sin \delta_2} \\
		c=\left\{
		\begin{array}{ccl} 
		\frac{q_1}{q_2} & \mathrm{if}~\mathrm{(S2.9a)}~\mathrm{and}~(\alpha_1+\beta_1+\gamma_1+\delta_1-2\pi)(\alpha_2+\beta_2+\gamma_2+\delta_2-2\pi)>0 \\
		-\frac{q_1}{q_2} & \mathrm{if}~\mathrm{(S2.9a)}~\mathrm{and}~(\alpha_1+\beta_1+\gamma_1+\delta_1-2\pi)(\alpha_2+\beta_2+\gamma_2+\delta_2-2\pi)<0 \\
		\frac{q_2}{q_1} & \mathrm{if}~\mathrm{(S2.9b)}~\mathrm{and}~(-\alpha_1+\beta_1-\gamma_1+\delta_1)(-\alpha_2+\beta_2-\gamma_2+\delta_2)>0 \\
		-\frac{q_2}{q_1} & \mathrm{if}~\mathrm{(S2.9b)}~\mathrm{and}~(-\alpha_1+\beta_1-\gamma_1+\delta_1)(-\alpha_2+\beta_2-\gamma_2+\delta_2)<0
		\end{array}
		\right. \\
		\end{gathered}
		\end{equation}
		\item [Type 4.5 Elliptic]
		\begin{equation}
		\begin{gathered}
		\alpha_j \pm \beta_j \pm \gamma_j \pm \delta_j \neq 0~(\mathrm{mod}~2\pi, 1 \le j \le 4) \\
		M_1=M_2, \quad p_1=p_2 \\
		\dfrac{\sin \alpha_1\sin \gamma_1}{\sin (\sigma_1-\alpha_1)\sin (\sigma_1-\gamma_1)}=\dfrac{\sin \alpha_2\sin \gamma_2}{\sin (\sigma_2-\alpha_2)\sin (\sigma_2-\gamma_2)} \\
		c=\left\{
		\begin{array}{ccl} 
		\frac{q_1}{q_2} & \mathrm{if}~(\alpha_1+\beta_1+\gamma_1+\delta_1-2\pi)(\alpha_2+\beta_2+\gamma_2+\delta_2-2\pi)>0 \\
		-\frac{q_1}{q_2} & \mathrm{if}~(\alpha_1+\beta_1+\gamma_1+\delta_1-2\pi)(\alpha_2+\beta_2+\gamma_2+\delta_2-2\pi)<0 
		\end{array}
		\right. \\
		\end{gathered}
		\end{equation}
	\end{description}
	\item[Type 5.1 Linearly conjugated antideltoids]
	\begin{equation}
	\begin{gathered}
	\alpha_2+\gamma_2=\beta_2+\delta_2=\pi, \quad \alpha_4+\gamma_4=\beta_4+\delta_4=\pi \\
	\alpha_1+\delta_1=\beta_1+\gamma_1=\pi, \quad \alpha_3+\beta_3=\gamma_3+\delta_3=\pi \\
	\kappa_4^2 \mu_1 = \lambda_3, \quad \kappa_4 \zeta_1=\kappa_2 \xi_3
	\end{gathered}
	\end{equation}
	\item[Type 5.2 Linearly conjugated conics]
	\begin{equation}
	\begin{gathered}
	\alpha_2+\gamma_2=\beta_2+\delta_2=\pi, \quad \alpha_4+\gamma_4=\beta_4+\delta_4=\pi \\
	\alpha_1+\beta_1+\gamma_1+\delta_1=2\pi, \quad \alpha_3+\beta_3+\gamma_3+\delta_3=2\pi \\
	q_3=|\kappa_2|p_1, \quad q_1=|\kappa_4|p_3, \quad t_1=\left\{
	\begin{array}{ccl} 
	t_3 & \mathrm{if}~\kappa_2\kappa_4>0 \\
	t_3+\pi & \mathrm{if}~\kappa_2\kappa_4<0 
	\end{array}
	\right. \\
	\end{gathered}
	\end{equation}
	\item[Type 5.3 Linearly conjugated elliptics]
	\begin{equation}
	\begin{gathered}
	\alpha_2+\gamma_2=\beta_2+\delta_2=\pi, \quad \alpha_4+\gamma_4=\beta_4+\delta_4=\pi \\
	\alpha_1 \pm \beta_1 \pm \gamma_1 \pm \delta_1 \neq 0~(\mathrm{mod}~2\pi), \quad \alpha_3 \pm \beta_3 \pm \gamma_3 \pm \delta_3 \neq 0~(\mathrm{mod}~2\pi) \\
	M_1=M_3 \\
	p_1=|\kappa_2|q_3, \quad p_3=|\kappa_4|q_1, \quad t_1=\left\{
	\begin{array}{ccl} 
	t_3 & \mathrm{if}~\kappa_2\kappa_4>0 \\
	t_3+2K_1 & \mathrm{if}~\kappa_2\kappa_4<0 
	\end{array}
	\right. \\
	\end{gathered}
	\end{equation}
	\item[Type 6.1 Conic-deltoids]
	\begin{equation}
	\begin{gathered}
	\alpha_2+\beta_2+\gamma_2+\delta_2=2\pi, \quad \alpha_3+\beta_3+\gamma_3+\delta_3=2\pi \\
	\alpha_1+\delta_1=\beta_1+\gamma_1=\pi \quad \alpha_4+\delta_4=\beta_4+\gamma_4=\pi \\
	q_2=q_3, \quad p_1=p_2, \quad p_3=p_4 \\
	i\log \left(\dfrac{\epsilon_1}{\epsilon_4}\sqrt{\dfrac{\mu_1}{\mu_4}}\right)=\pm t_1 \pm t_2
	\end{gathered}
	\end{equation}
	Any combination of signs is sufficient for the last condition.
	\item[Type 6.2 First orthodiagonal-isogram]
	\begin{equation}
	\begin{gathered}
	\cos \alpha_1 \cos \gamma_1= \cos \beta_1 \cos \delta_1 \\ 
	\alpha_2+\beta_2=\gamma_2+\delta_2=\pi, \quad \lambda_1=\lambda_2 \\
	\alpha_4+\delta_4=\beta_4+\gamma_4=\pi, \quad \mu_1=\mu_4 \\
	\alpha_3=\gamma_3, \quad \beta_3=\delta_3 \\
	\nu_1=\kappa_3 \xi_2 \zeta_4
	\end{gathered}
	\end{equation}
	\item[Type 6.3 Second orthodiagonal-isogram]
	\begin{equation}
	\begin{gathered}
	\cos \alpha_1 \cos \gamma_1= \cos \beta_1 \cos \delta_1 \\ 
	\alpha_2+\beta_2=\gamma_2+\delta_2=\pi, \quad \lambda_1=\lambda_2 \\
	\alpha_3=\beta_3, \quad \gamma_3=\delta_3, \quad \alpha_4+\gamma_4=\beta_4+\delta_4 \\
	\kappa_4^2 \mu_1 =\lambda_3 \quad \kappa_4 \nu_1 =\xi_2 \xi_3
	\end{gathered}
	\end{equation}
	\item[Type 6.4 Conic-isogram]
	\begin{equation}
	\begin{gathered}
	\alpha_2+\beta_2+\gamma_2+\delta_2=2\pi \\ 
	\alpha_1+\delta_1=\beta_1+\gamma_1=\pi, \quad p_1=p_2 \\
	\alpha_3+\beta_3=\gamma_3+\delta_3=\pi, \quad q_3=q_2 \\
	\alpha_4=\gamma_4, \quad \beta_4=\delta_4 \\
	\kappa_4=\dfrac{
	q_3 \xi_3}{2i\epsilon_1\sqrt{-\mu_1}e^{\pm it_2}}
	\end{gathered}
	\end{equation}
	Any choice of sign is sufficient for the last condition.
	\item[Type 6.5 Conic-antiisogram]
	\begin{equation}
	\begin{gathered}
	\alpha_2+\beta_2+\gamma_2+\delta_2=2\pi \\ 
	\alpha_1+\delta_1=\beta_1+\gamma_1=\pi, \quad p_1=p_2 \\
	\alpha_3+\beta_3=\gamma_3+\delta_3=\pi, \quad q_3=q_2 \\
	\alpha_4+\gamma_4=\beta_4+\delta_4=\pi \\
	\kappa_4=\dfrac{2
	}{iq_3 \xi_3\epsilon_1\sqrt{-\mu_1}e^{\pm it_2}}
	\end{gathered}
	\end{equation}
	Any choice of sign is sufficient for the last condition.
	\item[Type 6.6 (anti)deltoid-elliptic]
	\begin{equation}
	\begin{gathered}
	\alpha_1=\delta_1, \quad \beta_1=\gamma_1 \\ 
	\alpha_2+\delta_2=\beta_2+\gamma_2=\pi, \quad \lambda_1=\lambda_2 \\
	p_1 \neq p_2 \\
	\alpha_3 \pm \beta_3 \pm \gamma_3 \pm \delta_3 \neq 0~(\mathrm{mod}~2\pi), \quad \alpha_4 \pm \beta_4 \pm \gamma_4 \pm \delta_4 \neq 0~(\mathrm{mod}~2\pi) \\
	M_3=M_4<1, \quad p_3=p_4 \\
	\pm t_3 \pm t_4 = lK+i\dfrac{K'}{2} \\
	\end{gathered}
	\end{equation}
	where
	\begin{equation*}
	\begin{gathered}
	\left\{
	\begin{array}{ccl} 
	\mu_1=\dfrac{p_1^2}{k}, \quad \mu_2=\dfrac{p_2^2}{k}, \quad \zeta_1\zeta_2=\dfrac{2(1+k)}{k\sqrt{k}}p_1p_2& \mathrm{if}~l=0 \\
	\mu_1=-\dfrac{p_1^2}{k}, \quad \mu_2=-\dfrac{p_2^2}{k}, \quad \zeta_1\zeta_2=\dfrac{2i(1-k)}{k\sqrt{k}}p_1p_2& \mathrm{if}~l=1 \\
	\mu_1=\dfrac{p_1^2}{k}, \quad \mu_2=\dfrac{p_2^2}{k}, \quad \zeta_1\zeta_2=-\dfrac{2(1+k)}{k\sqrt{k}}p_1p_2& \mathrm{if}~l=2 \\
	\mu_1=-\dfrac{p_1^2}{k}, \quad \mu_2=-\dfrac{p_2^2}{k}, \quad \zeta_1\zeta_2=-\dfrac{2i(1-k)}{k\sqrt{k}}p_1p_2& \mathrm{if}~l=3
	\end{array}
	\right. \\
	\end{gathered}
	\end{equation*}
	Any combination of sign is sufficient for the last condition.
	\item[Type 6.7 isogram-deltoid-conic]
	\begin{equation}
	\begin{gathered}
	\alpha_3 +\beta_3+\gamma_3+\delta_3=2\pi \\ 
	\alpha_4+\delta_4=\beta_4+\gamma_4=\pi, \quad p_3=p_4 \\
	\alpha_1+\delta_1=\beta_1+\gamma_1=\pi \\
	\alpha_2+\gamma_2=\beta_2+\delta_2=\pi \\
	\left\{
	\begin{array}{ccl}	
	\mu_1=\mu_4e^{2it_3}, \quad \zeta_1=\dfrac{2i\kappa_2\epsilon_4\sqrt{-\mu_4}e^{it_3}}{q_2} \\
	\mathrm{or} \\
	\mu_1=\mu_4e^{-2it_3}, \quad \zeta_1=\dfrac{2i\kappa_2\epsilon_4\sqrt{-\mu_4}e^{-it_3}}{q_2} \\
	\end{array} 
	\right.
	\end{gathered}
	\end{equation}
	\item[Type 6.8 Three conics with isogram]
	\begin{equation}
	\begin{gathered}
	\alpha_j +\beta_j+\gamma_j+\delta_j=2\pi ~(j=1,3,4)\\  
	p_3=p_4, \quad q_1=q_4\\
	\alpha_2+\gamma_2=\beta_2+\delta_2=\pi \\
	q_3=|\kappa_2| p_1, \quad t_1 \pm t_3 \pm t_4 = \left\{
	\begin{array}{ccl}	
	0,~\mathrm{if}~\kappa_2>0 \\
	\pi,~\mathrm{if}~\kappa_2<0 \\
	\end{array} 
	\right.
	\end{gathered}
	\end{equation}
	Any combination of signs is sufficient for the last condition.
	\item[Type 6.9 Three elliptics with isogram]
	\begin{equation}
	\begin{gathered}
	\alpha_j \pm \beta_j \pm \gamma_j \pm \delta_j \neq 0~(\mathrm{mod}~2\pi,~ j=1,3,4) \\ 
	M_1=M_3=M_4 \\
	p_3=p_4, \quad q_1=q_4\\
	\alpha_2+\gamma_2=\beta_2+\delta_2=\pi \\
	p_1=|\kappa_2| q_3, \quad t_1 \pm t_3 \pm t_4 = \left\{
	\begin{array}{ccl}	
	0,~\mathrm{if}~\kappa_2>0 \\
	2K,~\mathrm{if}~\kappa_2<0 \\
	\end{array} 
	\right.
	\end{gathered}
	\end{equation}
	Any combination of signs is sufficient for the last condition.
	\item[Type 6.10 orthodiagonal-(anti)deltoid-conic]
	\begin{equation}
	\begin{gathered}
	\cos \alpha_1 \cos \gamma_1= \cos \beta_1 \cos \delta_1 \\ 
	\alpha_2+\beta_2=\gamma_2+\delta_2=\pi, \quad \lambda_1=\lambda_2 \\
	\alpha_3+\gamma_3=\beta_3+\delta_3 \\
	\alpha_4=\delta_4, \quad \beta_4=\gamma_4, \quad p_3=p_4 \\
		\left\{
	\begin{array}{ccl}	
	\mu_1=\mu_4e^{2it_3}, \quad \nu_1=\dfrac{2i\xi_2\epsilon_4\sqrt{-\mu_4}e^{it_3}}{q_3} \\
	\mathrm{or} \\
	\mu_1=\mu_4e^{-2it_3}, \quad \nu_1=\dfrac{2i\xi_2\epsilon_4\sqrt{-\mu_4}e^{-it_3}}{q_3} \\
	\end{array} 
	\right.
	\end{gathered}
	\end{equation}
	\item[Type 7 Trivial] A rigid-foldable Kokotsakis quadrilateral is called \textit{trivial} if at least one of its folding angles keeps constant in all branches of its rigid folding motion. Those trivial rigid-foldable Kokotsakis quadrilaterals are classified into the following 4 types.
	\begin{description}
	\item[Type 7.1 First Trivial]
		\begin{equation}
		\begin{gathered}
		\alpha_1+\delta_1=\beta_1+\gamma_1=\pi \\
		\alpha_2+\delta_2=\beta_2+\gamma_2=\pi \\
		\end{gathered}
		\end{equation}
		\item[Type 7.2 Second Trivial]
		\begin{equation}
		\begin{gathered}
		\alpha_1+\delta_1=\beta_1+\gamma_1=\pi \\
		\alpha_3+\delta_3=\beta_3+\gamma_3=\pi \\
		\end{gathered}
		\end{equation}
		\item[Type 7.3 Third Trivial]
		\begin{equation}
		\begin{gathered}
		\alpha_1+\delta_1=\beta_1+\gamma_1=\pi \\
		\alpha_4+\delta_4=\beta_4+\gamma_4=\pi \\
		\end{gathered}
		\end{equation}
		\item[Type 7.4 Fourth Trivial]
		\begin{equation}
		\begin{gathered}
		\alpha_1+\delta_1=\beta_1+\gamma_1=\pi \\
		\alpha_2+\delta_2=\beta_2+\gamma_2=\pi \\
		\alpha_3+\delta_3=\beta_3+\gamma_3=\pi \\
		\alpha_4+\delta_4=\beta_4+\gamma_4=\pi \\
		\end{gathered}
		\end{equation}
	\end{description}
\end{description}

Note that in a few places, we have revised some minor typographical errors in \cite{izmestiev_classification_2016}, which are in Types 2.1, 5.2, 6.6, 6.7, 6.10.

\bibliographystyle{unsrt}